\documentclass[a4paper,reqno]{amsart}

\usepackage{amsmath}
\usepackage{amsthm}
\usepackage{amssymb}
\usepackage{amscd} 
\usepackage{color}
\usepackage[dvipsnames]{xcolor}
\usepackage{bbm} 
\usepackage{mathtools,mathrsfs} 
\usepackage{esint}
\usepackage{gensymb}

\usepackage{times} 

\usepackage{verbatim} 

\usepackage{enumitem}

\newtheoremstyle{newremark}
  {5pt}
  {5pt}
  {\rmfamily}
  {}
  {\rmfamily\bf}
  {.}
  {.5em}
  {}

\newtheorem{theorem}{Theorem}
\newtheorem{lemma}[theorem]{Lemma}
\newtheorem{corollary}[theorem]{Corollary}
\newtheorem{proposition}[theorem]{Proposition}

\theoremstyle{newremark}
\newtheorem{remark}[theorem]{Remark}

\newtheorem*{definition*}{Definition} 
\newtheorem*{notations*}{Notations}

\numberwithin{theorem}{section}
\numberwithin{equation}{section}

\newcommand{\N}{\mathbb{N}} 
\newcommand{\R}{\mathbb{R}} 

\def\XXint#1#2#3{{%
\setbox0=\hbox{$#1{#2#3}{\int}$}
\vcenter{\hbox{$#2#3$}}\kern-.5\wd0}}

\newcommand{\veps}{\varepsilon}

\renewcommand{\leq}{\leqslant}
\renewcommand{\geq}{\geqslant}
\renewcommand{\subset}{\subseteq}
\renewcommand{\supset}{\supseteq}

\newcommand{\res}{\mathop{\hbox{\vrule height 7pt width .5pt depth 0pt
\vrule height .5pt width 6pt depth 0pt}}\nolimits}

\newcommand{\eps}{\varepsilon}

\begin{document}


\title[Phase field approximation of the  Steiner problem]{On a phase field approximation of the planar Steiner problem: Existence, regularity, and asymptotic of minimizers}

\author{Matthieu Bonnivard}
\address{Laboratoire Jacques-Louis Lions (CNRS UMR 7598), Universit\'e Paris Diderot, Paris, France}
\email{bonnivard@ljll.univ-paris-diderot.fr}

\author{Antoine Lemenant}
\address{Laboratoire Jacques-Louis Lions (CNRS UMR 7598), Universit\'e Paris Diderot, Paris, France}
\email{lemenant@ljll.univ-paris-diderot.fr}

\author{Vincent Millot}
\address{Laboratoire Jacques-Louis Lions (CNRS UMR 7598), Universit\'e Paris Diderot, Paris, France}
\email{millot@ljll.univ-paris-diderot.fr}



\begin{abstract}
In this article, we consider and analyse a small variant of a functional originally introduced in \cite{BLS,LS} to approximate the (geometric) planar Steiner problem. This functional depends on a small parameter $\eps>0$ and resembles  the (scalar) Ginzburg-Landau functional from phase transitions. In a first part, we prove existence and regularity of minimizers for this functional. Then we provide a detailed analysis of their behavior as $\eps\to0$, showing in particular  that sublevel sets Hausdorff converge to optimal Steiner sets. Applications to the average distance problem and optimal compliance are also discussed. 
\end{abstract}



\maketitle


\tableofcontents


\section{Introduction}


 In its simplest version, the original (planar) Steiner problem consists  in finding, for a given collection of points $a_0,...,a_N \in \R^2$, a compact connected set $K\subset \R^2$ containing all the $a_i$'s and having minimal length. From the geometric analysis point of view, the Steiner problem can be seen as the one dimensional version of the (unoriented) Plateau problem, which consists in finding a (unoriented) surface of least area spanning a given boundary. Solutions to the Steiner problem exist and are usually not unique. However, every solution consists of a finite tree made of straight segments joining by number of three with 120\degree angles. This rigid structure allows one to reduce the Steiner problem to a discrete problem, but finding an exact  solution is  known to be computationally very hard:  it belongs to the original list of NP-complete problems proposed by Karp \cite{Karp}. And, obviously, the discrete approach is unadapted if one considers a perturbed version of the problem as it may arise in some models from continuum mechanics. 
 These facts motivate the development of specific analytic/geometric tools, and more precisely of approximation procedures that can be numerically implemented.   
 
 Concerning minimal boundaries (boundaries of least area), the typical oriented Plateau problem, such approximations are well known by now, the most common ones being the so-called {\sl phase field approximations}. They usually rely on the minimization of  an energy functional based on the van der Waals-Cahn-Hilliard  theory for phase transitions (see e.g. \cite{Gurt,Mod,ModMort}), explaining the terminology.  Applications of phase field methods to unoriented problems are more recent. The first one might be the Ambrosio-Tortorelli method \cite{ATo1,ATo2} used to approximate the Mumford-Shah functional from image segmentation \cite{MumSh}. Nowadays, the Mumford-Shah functional receives a lot of interest from the materials science  community, and the Ambrosio-Tortorelli approximation is, for instance,   heavily used to simulate crack propagation in elastic solids \cite{BFM1,BFM2}. 
 
For a long time, no phase field methods (for unoriented Plateau type problems) were designed to include  topological constraints such as connectedness. Only recently such a method has been suggested, first in \cite{LS}, and then in  \cite{BLS}, to approximate the planar Steiner problem and/or related minimization problems involving the length of connected sets. In \cite{DLS} the same approach has been successfully implemented (theoretically and numerically) to approximate the Willmore energy of connected curves or surfaces.  
 At the present time, two alternative (but  complementary) methods to solves the Steiner problem just appeared as preprints \cite{BOO,chamb}. 
 
 The main objective of this article is to complement the analysis initiated in \cite{BLS,LS} in the following way.  Although the $\Gamma$-convergence result of \cite{BLS,LS} proves that ``some approximate minimization problems'' indeed approximate the Steiner problem (or variants), existence of minimizers for the underlying functionals cannot be proved (at least easily), nor qualitative properties of ``almost'' minimizers. This is essentially due to the analytical complexity in the construction of those functionals. Here we introduce a tiny variant of  \cite{BLS,LS} with great benefits. In few words, we are able to prove for the new functional existence and regularity of minimizers, as well as a more precise description of their behavior in the singular limit.  Before going further, let us describe our results in detail. 
\vskip3pt

Consider a bounded and convex open set $\Omega_0\subset\R^2$. Given a nonnegative Borel measurable function $w:\overline\Omega_0\to[0,\infty)$, we define the (generalized) geodesic distance between two points $a,b\in\overline\Omega_0$ relative to the conformal metric $w$ to be  
$${\bf D}(w;a,b):=\inf_{\Gamma: a \leadsto b}\int_{\Gamma}w\,d\mathcal{H}^1 \in [0,+\infty]\,,$$
where $\Gamma: a\leadsto b$ means that $\Gamma$ is a rectifiable curve in  $\overline\Omega_0$ of finite length connecting $a$ and $b$ (i.e., $\Gamma$ a Lipschitz image of $[0,1]$ contained in $\overline\Omega_0$ running from $a$ to $b$). 

We fix a positive finite measure $\mu$ supported on $\overline\Omega_0$, a {\sl base point} $a_0\in\overline\Omega_0$, and a bounded smooth open set $\Omega\subset \R^2$ such that $\overline\Omega_0\subset\Omega$. For a given set of parameters $\eps,\lambda_\eps,\delta_\eps\in(0,1)$,  we consider   
the functional $F_\eps^\mu:H^1(\Omega)\cap L^\infty(\Omega)\to [0,+\infty)$ defined by 
$$F_\eps^\mu(u):=\eps\int_\Omega|\nabla u|^2\,dx+\frac{1}{4\eps}\int_{\Omega}(u-1)^2\,dx+\frac{1}{\lambda_\eps}\int _{\overline\Omega_0}{\bf D}\big(\delta_\eps+u^2;a_0,x\big)\,d\mu\,,$$
where, in the ${\bf D}$-term, $\delta_\varepsilon+u^2$ denotes the {\sl precise representative}  of the Sobolev function $\delta_\varepsilon+u^2\in W^{1,1}(\Omega)\cap L^\infty(\Omega)$. In this way, the value of ${\bf D}\big(\delta_\eps+u^2;a_0,x\big)$ only depends on $a_0$, $x$, and the equivalence class of $\delta_\varepsilon+u^2$. Moreover, the function 
$x\mapsto {\bf D}\big(\delta_\eps+u^2;a_0,x\big)$ turns out to be  $(\delta_\varepsilon+\|u\|^2_{L^\infty(\Omega)})$-Lipschitz continuous (see Remark \ref{remgeodist}), so that $F_\eps^\mu$ is well defined (or more precisely, its last term). 
\vskip3pt

We are interested in the minimization problem 
\begin{equation}
\min_{u\in 1+H^1_0(\Omega)\cap L^\infty(\Omega)} F^\mu_\eps(u)\,. \label{problemM1}
\end{equation}
Our first main result deals with existence and regularity of solutions. 

\begin{theorem}\label{thmmain1}
Problem \eqref{problemM1} admits at least one solution. In addition, any solution $u_\varepsilon$ belongs to $W^{1,p}(\Omega)$ for every $p<\infty$ (in particular, $u_\varepsilon\in C^{0,\alpha}(\Omega)$ for every $\alpha\in(0,1)$), and $0\leq u_\varepsilon\leq 1$. 
\end{theorem}

Let us mention that the regularity above is essentially sharp in the sense that $u_\varepsilon$ is in general {\sl not} Lipschitz continuous globally in $\Omega$ (see Remarks \ref{counterex} \& \ref{regC1alphaCurve}). In the case where ${\rm spt}\,\mu$ is finite, we shall see that $u_\varepsilon$ is in fact $C^\infty$ away from finitely many $C^{1,\alpha}$-curves connecting $a_0$ to ${\rm spt}\,\mu$ (given by minimizing geodesics for the distance ${\bf D}\big(\delta_\eps+u_\eps^2)$).  
\vskip3pt

We now describe the asymptotic behavior of minimizers of $F^\mu_\eps$ as $\varepsilon\to 0$. For this issue, we shall assume (for simplicity) that the two parameters 
$\lambda_\varepsilon$ and $\delta_\varepsilon$ satisfy the following relation: 
\begin{equation}\label{deppara}
\lambda_\varepsilon\mathop{\longrightarrow}\limits_{\varepsilon\to 0} 0\quad\text{and}\quad \delta_{\varepsilon}=\lambda^\beta_\varepsilon\;\text{ for some $\beta\in(1,2)$}\,.
\end{equation}
Provided that $\mathcal{H}^1({\rm spt}\,\mu)<\infty$, our second main result shows that sublevel sets of minimizers converge to a solution of the generalized Steiner probem
\begin{equation}\label{classStein}
\min\Big\{\mathcal{H}^1(K): K\subset\R^2 \text{ compact and connected, } K\supset\{a_0\}\cup{\rm spt}\,\mu \Big\} \,.
\end{equation}
Note that for $\mu=\sum_{i=0}^N\delta_{a_i}$ and some distinct points $a_i\in\overline\Omega_0$, problem \eqref{classStein} coincides with the classical Steiner problem described previously.

\begin{theorem}\label{thmmain2} 
Assume that ${\rm spt}\,\mu$ is not reduced to $\{a_0\}$ and that $\mathcal{H}^1({\rm spt}\,\mu)<\infty$.  Assume also that \eqref{deppara} holds. Let $\varepsilon_k\downarrow0$ and $\{u_k\}_{k\in\mathbb{N}}\subset 1+H^1_0(\Omega)$ be such that 
$$ F^\mu_{\varepsilon_k}(u_k)=\min_{1+H_0^1(\Omega)}F^\mu_{\varepsilon_k}\quad\text{for each $k\in\mathbb{N}$}\,.$$
There exist a (not relabeled) subsequence and a compact connected set $K_*\subset \overline\Omega_0$ such that $\{u_k \leq t \}\to K_*$ in the Hausdorff  sense for every $t\in(0,1)$. 
In addition, $K_*$ solves the Steiner problem \eqref{classStein} relative to $\{a_0\}\cup{\rm spt}\,\mu\,$,  and the following holds: 
\begin{enumerate}
\item[(i)] $F_{\varepsilon_k}^\mu(u_k)\to \mathcal{H}^1(K_*)$; 
\vskip3pt

\item[(ii)] ${\bf D}\big(\delta_{\varepsilon_k}+u_k^2 ;a_0,x\big)\to {\rm dist}(x,K_*)$ uniformly on $\overline\Omega_0$;
\vskip3pt

\item[(iii)]  $u_k\to 1$ in $C^{2}_{\rm loc}(\overline\Omega\setminus K_*)$. 
\end{enumerate}
\end{theorem}

In proving this theorem, we make use of the main result in \cite{BLS,LS} that we now briefly present. The original functional introduced in  \cite{BLS,LS}   is (essentially)  
$\widetilde F^\mu_\eps:1+H_0^1(\Omega)\cap C^0(\overline\Omega)\to [0,\infty)$ given by
\begin{equation}\label{defftild}
\widetilde F^\mu_\eps(u):=\begin{cases}
\displaystyle \eps\int_\Omega|\nabla u|^2\,dx+\frac{1}{4\eps}\int_\Omega(1-u)^2\,dx+\frac{1}{\lambda_\eps}\int_{\overline\Omega_0} {\bf D}(u;a_0,x)\,d\mu & \text{if $0\leq u\leq 1$}\,,\\[8pt]
+\infty & \text{otherwise}\,.
\end{cases}
\end{equation}
As explained \cite[Section 5.4]{BLS}, the possible lack of lower semicontinuity of $\widetilde F^\mu_\eps$ prevents one to prove existence of minimizers (at least easily -- and existence is still unknown\footnote{We learned from Dorin Bucur that the recent preprint \cite{Bucur} contains results solving some lower semicontinuity issues in a similar direction.}).  
The main result of \cite{BLS,LS} is of $\Gamma$-convergence nature, and shows the two following  facts: {\it (1)} if a sequence $\{v_\eps\}$ satisfies $\widetilde F^\mu_\eps(v_\eps)=O(1)$, then $x\mapsto {\bf D}(v_\eps;a_0,x)$ (sub-)converges uniformly as $\eps\to 0$ to some function ${\bf d}_*$, $\{{\bf d}_*=0\}$ is a compact connected set containing $\{a_0\}\cup{\rm spt}\,\mu$, and $\mathcal{H}^1(\{{\bf d}_*=0\})\leq \liminf_\eps \widetilde F^\mu_\eps(v_\eps)$;  {\it (2)} for every compact connected set $K$ containing $\{a_0\}\cup{\rm spt}\,\mu$, there exists a sequence $\{w_\eps\}$ of functions of finite $\widetilde F^\mu_\eps$-energy  satisfying  $\limsup_\eps \widetilde F^\mu_\eps(w_\eps)\leq \mathcal{H}^1(K)$. In particular, if the sequence $\{v_\eps\}$ is ``almost'' minimizing in the sense that $\widetilde F^\mu_\eps(v_\eps)-\inf \widetilde F^\mu_\eps=o(1)$, then the set $\{{\bf d}_*=0\}$ solves the Steiner problem \eqref{classStein}, and $\widetilde F^\mu_\eps(v_\eps)\to \mathcal{H}^1(\{{\bf d}_*=0\})$. 

In conclusion, the main contribution of Theorem \ref{thmmain2} is the Hausdorff convergence of the sublevel sets $\{u_\eps\leq t\}$, the convergence estimate away from the limiting Steiner set, and the identification of the limiting function ${\bf d}_*$. Compare to $\widetilde F^\mu_\eps$, this is made possible by introducing the additional parameter $\delta_\eps$ and replacing $u$ by $u^2$ in the ${\bf D}$-term.  The parameter $\delta_\eps$, already suggested in \cite{BLS}, can be seen as an {\sl elliptic regularisation term}. In turn, the term $u^2$ is the key new ingredient which allows to get a {\sl linear elliptic equation} for $u_\eps$ (at least if ${\rm spt}\,\mu$ is finite). A large part of the arguments used to prove both Theorem \ref{thmmain1} and Theorem \ref{thmmain2} rests on this equation and rather classical {\sl linear estimates}. The introduction of the ``safety zone'' $\Omega\setminus\overline\Omega_0$ (not present in \cite{BLS})  is just a convenient way to avoid boundary effects, and has no other importance.  Finally, we impose relation  \eqref{deppara} between $\lambda_\eps$ and $\delta_\eps$ for the following reason: on one hand the condition $\delta_\eps=o(\lambda_\eps)$ is necessary to derive the Steiner problem in the limit; on the other hand the condition $\lambda^2_\eps=o(\delta_\eps)$ allows us to use \cite{BLS} in a straightforward way, even if it is probably unnecessary. 
\vskip3pt

We close this introduction mentioning our companion paper \cite{BLM2}, second part of our work, where we consider the minimization of a discretized version $F^\mu_\eps$ based on finite $\mathcal{P}^1$-elements. A special attention will be devoted on how to handle the ${\bf D}$-term in this discrete framework. 
Using the material of this paper, we will  be able to determine explicit estimates on the grid size in terms of $\eps$ to ensure the convergence of discrete  minimizers  to Steiner sets,  in the spirit of Theorem  \ref{thmmain2}. 
\vskip3pt

This paper is organized as follows. In Section \ref{newsection},  we consider the case where $\mu$ has a finite support. We start establishing a priori estimates leading to existence and (as a byproduct) regularity of minimizers (see Corollary \ref{thmminF}). The case of a general measure $\mu$ is treated in Section \ref{genmeassec} through an approximation argument using  finitely supported measures. In Subsection \ref{AverdistCompsec}, we apply our existence theory for $F^\mu_\eps$ to prove existence of minimizers for functionals introduced in \cite{BLS} (and accordingly modified here) to approximate the average distance and compliance problems. Theorem \ref{thmmain2} is finally proved in Section \ref{sectasympt}.


\section{Existence and regularity for measures with finite support} \label{newsection}     


Throughout this section, we assume that the  measure $\mu$ has finite support, i.e.,  
\begin{equation}\label{discrmeas}
\mu=\sum_{i=1}^N\beta_i\,\delta_{a_i} 
\end{equation}
for some distinct points $a_1,\ldots,a_N\in\overline{\Omega}_0$ and coefficients $\beta_i>0$.  We fix a {\bf base point} $a_0\in\overline\Omega_0$ (possibly equal to one of the $a_i$'s), and to the resulting collection of points, we associate the following space of Lipschitz curves
$$\mathscr{P}(a_0,\mu):=\Big\{\overrightarrow{\boldsymbol{\gamma}}=(\gamma_i)_{i=1}^N: \gamma_i\in \mathscr{P}(a_0,a_i)\Big\}\,,$$ 
where we have set 
$$\mathscr{P}(a,b):=\Big\{\gamma\in{\rm Lip}([0,1];\overline\Omega_0): \gamma(0)=a\text{ and }\gamma(1)=b\Big\}\,. $$
We endow $\mathscr{P}(a_0,\mu)$ with the topology of uniform convergence. In this way, $\mathscr{P}(a_0,\mu)$ appears to be a subset of the complete metric space   $[C^0([0,1];\overline\Omega_0)]^N$. For  $\overrightarrow{\boldsymbol{\gamma}}\in \mathscr{P}(a_0,\mu)$, we write 
$$\Gamma(\gamma_i):=\gamma_i([0,1])\quad\text{and}\quad\Gamma({\overrightarrow{\boldsymbol{\gamma}}}):=\bigcup_{i=1}^N\gamma_i([0,1])\,. $$
For a given $\overrightarrow{\boldsymbol{\gamma}}\in \mathscr{P}(a_0,\mu)$, we consider the functional $E^\mu_\varepsilon(\cdot,\overrightarrow{\boldsymbol{\gamma}}):H^1(\Omega)\to [0,+\infty]$ defined by 
\begin{equation}\label{defEmu}
E^\mu_\varepsilon(u,\overrightarrow{\boldsymbol{\gamma}}):=\eps\int_\Omega|\nabla u|^2\,dx+\frac{1}{4\eps}\int_{\Omega}(u-1)^2\,dx+\frac{1}{\lambda_\eps}\sum_{i=1}^N\beta_i\int_{\Gamma(\gamma_i)}(\delta_\varepsilon+u^2)\,d\mathcal{H}^1\,,
\end{equation}
where each term $\int_{\Gamma(\gamma_i)}(\delta_\varepsilon+u^2)\,d\mathcal{H}^1$ is understood as the integration of the precise representative of $\delta_\varepsilon+u^2$ with respect to the measure $\mathcal{H}^1\res \Gamma(\gamma_i)$, see Subsection \ref{SecPrecRep} below. 

By the very definition of $F^\mu_\varepsilon$, the functional $E^\mu_\varepsilon$ relates to $F^\mu_\varepsilon$ through the formula 
\begin{equation}\label{relEF}
F^\mu_\varepsilon(u)=\inf_{\overrightarrow{\boldsymbol{\gamma}}\in \mathscr{P}(a_0,\mu)}E^\mu_\varepsilon(u,\overrightarrow{\boldsymbol{\gamma}})\qquad \forall u\in H^1(\Omega)\cap L^\infty(\Omega)\,.
\end{equation}
As we shall see, this identity is the key ingredient to investigate existence and regularity of minimizers of $F^\mu_\varepsilon$. In the same spirit, we also consider the functional 
  $G^\mu_\varepsilon:\mathscr{P}(a_0,\mu)\to [0,+\infty)$ defined by
\begin{equation}\label{defGmu}
G^\mu_\varepsilon(\overrightarrow{\boldsymbol{\gamma}}):= \inf_{u\in 1+H^1_0(\Omega)}  
 E^\mu_\varepsilon(u,\overrightarrow{\boldsymbol{\gamma}}) \,,
\end{equation}
and prove existence of minimizers. 

\subsection{The precise representative of a Lebesgue function}\label{SecPrecRep} 

The object of this subsection is to summarize some basic facts concerning the precise representative of a function, and their implications for the generalized geodesic distance. In doing so, we consider an open set $U\subset \mathbb{R}^n$. For $v\in L^1_{\rm loc}(U)$, the value of the precise representative of $v$ at $x\in U$ is defined by 
$$v^*(x):=\begin{cases}
\displaystyle \lim_{r\downarrow0}\fint_{B(x,r)}v(y)\,dy & \text{if the limit exists}\,,\\[5pt]
0 & \text{otherwise}\,.
\end{cases}$$
The pointwise defined function $v^*$ only depends on the equivalence class of $v$, and $v^*=v$ a.e. in $U$. 
In turn, we say that $v$ has an approximate limit at $x$ if there exists $t\in\mathbb{R}$ such that 
\begin{equation}\label{lebpt}
\lim_{r\downarrow0}\fint_{B(x,r)}|v(y)-t|\,dy=0\,.
\end{equation}
The set $S_v$ of points where this property fails is called the approximate discontinuity set. It is a $\mathcal{L}^n$-negligible Borel set, and for $x\in U$ the value $t$ determined by \eqref{lebpt} is equal to $v^*(x)$. In addition, the Borel function $v^*:\Omega\setminus S_v\to\mathbb{R}$  is approximately continuous at every point $x\in U\setminus S_v$ (see e.g. \cite[Section 3.6]{AFP} and \cite[Section 1.7.2]{EvGa}). 

We shall make use of the following elementary properties: 
\begin{enumerate}
\item[(i)] if $v_1\leq v_2$ a.e. in $U$, then $v_1^*(x)\leq v_2^*(x)$ for every $x\in U\setminus (S_{v_1}\cup S_{v_2})$;
\vskip3pt

\item[(ii)] if $f:\mathbb{R}\to\mathbb{R}$ is a Lipschitz function and $w:=f\circ v$, then $S_w\subset S_v$ and $w^*(x)=f(v^*(x))$ for every $x\in \Omega\setminus S_v$. 
\end{enumerate}

Finally, by standard results on $BV$-functions (see \cite[Section 3.7]{AFP}), we have $\mathcal{H}^{n-1}(S_v)=0$ whenever $v\in W^{1,1}_{\rm loc}(U)$. In what follows, we may  write $v$ instead of $v^*$  if it is clear from the context. 

\begin{remark}\label{remgeodist}
For a nonnegative  $v\in W^{1,1}_{\rm loc}(U)\cap L^\infty(U)$, one has $0\leq v^*(x)\leq \|v\|_{L^\infty(U)}$ at every point $x\in U\setminus S_v$, as a consequence of (i) above. In particular, 
$$0\leq \int_{\Gamma}v\,d\mathcal{H}^1\leq \|v\|_{L^\infty(U)} \mathcal{H}^1(\Gamma)$$
for every rectifiable curve $\Gamma\subset U$. As a consequence, if $U$ is assumed to be convex, one has 
$$0\leq {\bf D}(v;a,b):= \inf_{\Gamma: a \leadsto b}\int_{\Gamma}v\,d\mathcal{H}^1\leq  \|v\|_{L^\infty(U)}|a-b|\qquad \forall a,b\in U\,,$$
where the infimum is taken over all rectifiable curves $\Gamma\subset U$ running from $a$ to $b$. It is then customary to prove that the function $x\mapsto {\bf D}(v;a,x)$ is $\|v\|_{L^\infty(U)}$-Lipschitz continuous. 
\end{remark}


\subsection{The minimization problem with prescribed curves}\label{subsecprescrcurv}

In this subsection, we investigate the minimization problem 
\begin{equation}\label{minprescrcurv}
\min_{u\in 1+H^1_0(\Omega)} E^\mu_\varepsilon(u,\overrightarrow{\boldsymbol{\gamma}})
\end{equation}
for a prescribed set of curves $\overrightarrow{\boldsymbol{\gamma}}$ satisfying a mild regularity constraint: we shall assume that it belongs to 
$$\mathscr{P}_\Lambda(a_0,\mu):=\Big\{\overrightarrow{\boldsymbol{\gamma}}\in  \mathscr{P}(a_0,\mu): {\bf Al}\big(\Gamma(\gamma_i)\big)\leq \Lambda\text{ for each $i$}\Big\} \,,$$
for a given constant $\Lambda\geq 2$, where we have set  
$${\bf Al}(K):=\sup\left\{\frac{\mathcal{H}^1(K\cap B(x,r))}{r}: r>0\,,\;x\in K\right\}\quad\text{for a closed set $K\subset \mathbb{R}^2$}\,.$$
In this context, we establish existence and uniqueness of the solution, as well as regularity estimates. The introduction of this regularity constraint is motivated by  the following lemma, consequence of a classical result due to N.G. Meyers \& W.P. Ziemer \cite{MeZi}.  

\begin{lemma}\label{lemfond}
If $\overrightarrow{\boldsymbol{\gamma}}\in\mathscr{P}_\Lambda(a_0,\mu)$, then 
the functional 
$$B_\mu[\overrightarrow{\boldsymbol{\gamma}}]:(u,v)\in H^1(\Omega)\times H^1(\Omega)\mapsto \sum_{i=1}^N\beta_i\int_{\Gamma(\gamma_i)}uv\,d\mathcal{H}^1$$
defines a symmetric, nonnegative, and continuous  bilinear form on $H^1(\Omega)$ satisfying
$$\big\|B_\mu[\overrightarrow{\boldsymbol{\gamma}}]\big\|\leq  C_\Omega\|\mu\| \Lambda\,,$$
for some constant $C_\Omega$ depending only on $\Omega$.
 \end{lemma}
 
 \begin{proof}
{\it Step 1.} For a given $i\in\{1,\ldots,N\}$, we consider the finite measure on $\mathbb{R}^2$ defined by $\mu_i:=\mathcal{H}^1\res \Gamma(\gamma_i)$. Let $x\in\mathbb{R}^2$ and $r>0$ such that $\Gamma(\gamma_i)\cap B(x,r)\not=\emptyset$. Choose a point $z\in \Gamma(\gamma_i)\cap B(x,r)$, and notice that $\Gamma(\gamma_i)\cap B(x,r)\subset \Gamma(\gamma_i)\cap B(z,2r)$. Then,
 $$\mu_i\big(B(x,r)\big)\leq \mu_i\big(B(z,2r)\big)\leq 2r{\bf Al}\big(\Gamma(\gamma_i)\big)\,,$$
 which shows that 
 $$\sup\left\{\frac{\mu_i\big(B(x,r)\big)}{r}: r>0\,,\;x\in\mathbb{R}^2\right\}\leq 2\Lambda\,. $$
Since $W^{1,1}(\mathbb{R}^2)$-functions are approximately continuous $\mathcal{H}^{1}$-a.e. in $\mathbb{R}^2$, we can apply \cite[Theorem 5.12.4]{Zi} (see also \cite{MeZi}) to infer that $w\in L^1(\mu_i)$ for every $w\in W^{1,1}(\mathbb{R}^2)$ (or more precisely, $w^*\in L^1(\mu_i)$), with the estimate
\begin{equation}\label{estiZiem}
\int_{\Gamma(\gamma_i)}|w|\,d\mathcal{H}^1=\int_{\mathbb{R}^2}|w|\,d\mu_i\leq C\Lambda\int_{\mathbb{R}^2}|\nabla w|\,dx\,,
\end{equation}
for some universal constant $C>0$.
\vskip3pt

\noindent{\it Step 2.} Let $u\in H^1(\Omega)\mapsto \bar u \in H^1(\mathbb{R}^2)$ be a continuous linear extension operator (whose existence is ensured by the smoothness of $\Omega$). Note that for $u,v\in H^1(\Omega)$, we have $\bar u \bar v \in W^{1,1}(\mathbb{R}^2)$. Since $\sum_i\beta_i=\mu(\Omega)$, it follows from Step 1 that $\bar u \bar v \in L^1(\mu_i)$ 
for each $i\in\{1,\ldots,N\}$ (or more precisely, $(\bar u \bar v)^* \in L^1(\mu_i)$), and 
\begin{multline*}
\big|B_\mu[\overrightarrow{\boldsymbol{\gamma}}](u,v)\big| \leq C\|\mu\|\Lambda \int_{\mathbb{R}^2}|\nabla(\bar u\bar v)|\,dx\\
\leq C\|\mu\|\Lambda \|\bar u\|_{H^1(\mathbb{R}^2)}\|\bar v\|_{H^1(\mathbb{R}^2)}\leq  C_\Omega\|\mu\| \Lambda \| u\|_{H^1(\Omega)}\| v\|_{H^1(\Omega)}\,,
\end{multline*}
which completes the proof. 
 \end{proof}

Given $\overrightarrow{\boldsymbol{\gamma}}\in\mathscr{P}_\Lambda(a_0,\mu)$, we now rewrite for $u\in H^1(\Omega)$, 
$$ E^\mu_\varepsilon(u,\overrightarrow{\boldsymbol{\gamma}})= \eps\int_\Omega|\nabla u|^2\,dx+\frac{1}{4\eps}\int_{\Omega}(u-1)^2\,dx+\frac{1}{\lambda_\eps}B_\mu[\overrightarrow{\boldsymbol{\gamma}}](u,u)+\frac{\delta_\varepsilon}{\lambda_\varepsilon}\sum_{i=1}^N\beta_i\mathcal{H}^1(\Gamma(\gamma_i))\,.$$
By the previous lemma, $E^\mu_\varepsilon(u,\overrightarrow{\boldsymbol{\gamma}})<\infty$ for every $u\in H^1(\Omega)$, 
and  $E^\mu_\varepsilon(\cdot,\overrightarrow{\boldsymbol{\gamma}})$ is lower semicontinuous with respect to weak convergence in $H^1(\Omega)$. Owing to the strict convexity of the functional $E^\mu_\varepsilon(\cdot,\overrightarrow{\boldsymbol{\gamma}})$, we conclude to the following
 
 \begin{theorem}\label{existuniqthmgammafix}
 Given $\overrightarrow{\boldsymbol{\gamma}}\in\mathscr{P}_\Lambda(a_0,\mu)$,  problem \eqref{minprescrcurv} 
admits a unique solution $u_{\overrightarrow{\boldsymbol{\gamma}}}$. 
\end{theorem}

For $\overrightarrow{\boldsymbol{\gamma}}\in\mathscr{P}_\Lambda(a_0,\mu)$, we shall refer to $u_{\overrightarrow{\boldsymbol{\gamma}}}$ as {\bf the potential of }$\overrightarrow{\boldsymbol{\gamma}}$. It satisfies the Euler-Lagrange  equation 
\begin{equation}\label{eqEL}
\begin{cases}
\displaystyle-\varepsilon^2\Delta u_{\overrightarrow{\boldsymbol{\gamma}}}= \frac{1}{4}(1-u_{\overrightarrow{\boldsymbol{\gamma}}}) -\frac{\varepsilon}{\lambda_\varepsilon}B_\mu[\overrightarrow{\boldsymbol{\gamma}}](u_{\overrightarrow{\boldsymbol{\gamma}}},\cdot) &\text{in $H^{-1}(\Omega)$}\,,\\[8pt]
u_{\overrightarrow{\boldsymbol{\gamma}}}=1 & \text{on $\partial\Omega$}\,.
\end{cases}
\end{equation}
Our next objective is to obtain some regularity estimates on $u_{\overrightarrow{\boldsymbol{\gamma}}}$ with explicit dependence on the parameters. We start with an elementary   lemma.

\begin{lemma}\label{bound1}
Let $\overrightarrow{\boldsymbol{\gamma}}\in\mathscr{P}_\Lambda(a_0,\mu)$. The potential $u_{\overrightarrow{\boldsymbol{\gamma}}}$ satisfies $0\leq u_{\overrightarrow{\boldsymbol{\gamma}}}\leq 1$ a.e. in $\Omega$, and $u_{\overrightarrow{\boldsymbol{\gamma}}}\in C^\infty\big(\overline\Omega\setminus\Gamma(\overrightarrow{\boldsymbol{\gamma}})\big)$. 
\end{lemma}

\begin{proof}
Let us first prove that $0\leq u_{\overrightarrow{\boldsymbol{\gamma}}}\leq 1$ a.e. in $\Omega$. To this purpose, we consider the Lipschitz function $f(t):=\max(\min(t,1),0)$, and the competitor $v:=f\circ u_{\overrightarrow{\boldsymbol{\gamma}}}$.  It is a classical fact that $v\in 1+H_0^1(\Omega)$, and $|\nabla v|\leq |\nabla u_{\overrightarrow{\boldsymbol{\gamma}}}|$ a.e. in $\Omega$. 
Since  $u^2_{\overrightarrow{\boldsymbol{\gamma}}}$ belongs to $W^{1,1}(\Omega)$, we also have  
$f\circ u^2_{\overrightarrow{\boldsymbol{\gamma}}} \in W^{1,1}(\Omega)$. Noticing that  $v^2\leq f\circ u^2_{\overrightarrow{\boldsymbol{\gamma}}} $ a.e. in $\Omega$, we derive that 
$$(v^2)^*(x)\leq  \big(f\circ u^2_{\overrightarrow{\boldsymbol{\gamma}}} \big)^*(x) =f\big((u^2_{\overrightarrow{\boldsymbol{\gamma}}})^*(x)\big)\leq (u^2_{\overrightarrow{\boldsymbol{\gamma}}})^*(x)\quad\text{for every $x\in\Omega\setminus(S_{v^2}\cup S_{u^2_{\overrightarrow{\boldsymbol{\gamma}}}})$}\,.$$
Consequently, $(v^2)^*\leq (u^2_{\overrightarrow{\boldsymbol{\gamma}}})^*$ $\mathcal{H}^1$-a.e. in $\Omega$, so that $B_\mu[\overrightarrow{\boldsymbol{\gamma}}](v,v)\leq B_\mu[\overrightarrow{\boldsymbol{\gamma}}](u_{\overrightarrow{\boldsymbol{\gamma}}},u_{\overrightarrow{\boldsymbol{\gamma}}})$. 

From this discussion, we easily infer that $E^\mu_\varepsilon(v,\overrightarrow{\boldsymbol{\gamma}})\leq E^\mu_\varepsilon(u_{\overrightarrow{\boldsymbol{\gamma}}},\overrightarrow{\boldsymbol{\gamma}})$ with strict inequality if $\{v\not= u_{\overrightarrow{\boldsymbol{\gamma}}}\}$ has a non vanishing Lebesgue measure. Hence the conclusion follows from the minimality of  $u_{\overrightarrow{\boldsymbol{\gamma}}}$. 

Now we observe that $u_{\overrightarrow{\boldsymbol{\gamma}}}\in H^1(\Omega)\cap L^\infty(\Omega)$ satisfies
$$-\varepsilon^2\Delta u_{\overrightarrow{\boldsymbol{\gamma}}}= \frac{1}{4}(1-u_{\overrightarrow{\boldsymbol{\gamma}}}) \quad\text{in $\mathscr{D}^\prime\big(\Omega\setminus \Gamma(\overrightarrow{\boldsymbol{\gamma}})\big)$}\,. $$
From this equation and \eqref{eqEL}, we conclude that  $u_{\overrightarrow{\boldsymbol{\gamma}}}\in C^\infty\big(\overline\Omega\setminus\Gamma(\overrightarrow{\boldsymbol{\gamma}})\big)$ by means of the standard elliptic regularity theory for bounded weak solutions (see e.g. \cite{GiTr}).
\end{proof}

\begin{lemma}\label{comp}
Let $\overrightarrow{\boldsymbol{\gamma}}\in\mathscr{P}_\Lambda(a_0,\mu)$. At every $x_0\in \Omega\setminus \Gamma(\overrightarrow{\boldsymbol{\gamma}})$ satisfying ${\rm dist}(x_0,\Gamma(\overrightarrow{\boldsymbol{\gamma}}))\geq 12\varepsilon$, we have
$$0\leq 1- u_{\overrightarrow{\boldsymbol{\gamma}}}(x_0)\leq \exp\left(-\frac{3\,{\rm dist}(x_0,\Gamma(\overrightarrow{\boldsymbol{\gamma}}))}{32\varepsilon}\right)\,.$$
\end{lemma}

\begin{proof}
Set $R:=\frac{3}{4}{\rm dist}\big(x_0,\Gamma(\overrightarrow{\boldsymbol{\gamma}})\big)\geq 9\varepsilon$. We consider the function $v:=1-u_{\overrightarrow{\boldsymbol{\gamma}}}$ which satisfies $0\leq v\leq 1$ and solves 
$$ 
\begin{cases}
-4\varepsilon^2\Delta v+v=0 &\text{in $B(x_0,R)\cap \Omega$}\,,\\
v=0 & \text{on $ B(x_0,R)\cap \partial\Omega$}\,.
\end{cases}
$$
Now we introduce the function
$$\omega(x):=\exp\left(\frac{|x-x_0|^2-R^2}{8\varepsilon R}\right)\,.$$
As in \cite[Lemma 2]{BBH1}, our choice of $R$ implies that $\omega$ satisfies
$$\begin{cases}
-4\varepsilon^2\Delta\omega+\omega\geq 0 & \text{in $B(x_0,R)\cap \Omega$}\,,\\
\omega=1 & \text{on $\partial B(x_0,R)\cap \Omega$}\,,\\
\omega\geq 0 & \text{on $ B(x_0,R)\cap \partial\Omega$}\,.
\end{cases}$$
Then we infer from the maximum principle that $v\leq \omega$ in $B(x_0,R)\cap \Omega$. Evaluating this inequality at $x_0$ leads to the announced inequality.
\end{proof}

We now provide some pointwise estimates for the first and second derivatives of $u_{\overrightarrow{\boldsymbol{\gamma}}}$. 
Usefulness of these explicit estimates will be revealed in the second part of our work \cite{BLM2}.

\begin{lemma}\label{estigradhessloin}
Let $\overrightarrow{\boldsymbol{\gamma}}\in\mathscr{P}_\Lambda(a_0,\mu)$. At every $x_0\in \overline\Omega\setminus \Gamma(\overrightarrow{\boldsymbol{\gamma}})$ satisfying ${\rm dist}(x_0,\Gamma(\overrightarrow{\boldsymbol{\gamma}}))\geq 13\varepsilon$, we have
$$\big|\nabla u_{\overrightarrow{\boldsymbol{\gamma}}}(x_0)\big| \leq \frac{C_{\boldsymbol{\eta}_0}}{\varepsilon}  \exp\left(-\frac{{\rm dist}(x_0,\Gamma(\overrightarrow{\boldsymbol{\gamma}}))}{32\varepsilon}\right)\,,$$
and 
$$\big|\nabla^2 u_{\overrightarrow{\boldsymbol{\gamma}}}(x_0)\big| \leq \frac{C_{\boldsymbol{\eta}_0}}{\varepsilon^2}  \exp\left(-\frac{{\rm dist}(x_0,\Gamma(\overrightarrow{\boldsymbol{\gamma}}))}{32\varepsilon}\right)\,,$$
for some constant $C_{\boldsymbol{\eta}_0}$ depending only on $\Omega$ and $\boldsymbol{\eta}_0:=\min\big\{{\rm dist}(z,\Omega_0): z\in\partial\Omega\big\}>0$. 
\end{lemma}

\begin{proof}
{\it Step 1 (Interior estimates).} We assume in this step that $B(x_0,\varepsilon)\subset \Omega$. Define for $x\in B_1$, the function $w_\varepsilon:=1-u_{\overrightarrow{\boldsymbol{\gamma}}}(x_0+\varepsilon x)$. Then, $w_\varepsilon$ solves 
\begin{equation}\label{eqrescw}
-\Delta w_\varepsilon=\frac{1}{4}w_\varepsilon \quad\text{in $B_1$}\,.
\end{equation}
By Lemma \ref{comp}, we have for every $x\in B_1$, 
$$0\leq w_\varepsilon(x)\leq  \exp\left(-\frac{3\,{\rm dist}(x_0+\varepsilon x,\Gamma(\overrightarrow{\boldsymbol{\gamma}}))}{32\varepsilon}\right)\leq C \exp\left(-\frac{3\,{\rm dist}(x_0,\Gamma(\overrightarrow{\boldsymbol{\gamma}}))}{32\varepsilon}\right)\,.$$
Then we infer from \eqref{eqrescw} and \cite[Theorem 3.9]{GiTr} that 
\begin{equation}\label{gradbd}
|\nabla w_\varepsilon(x)|\leq C\|w_\varepsilon\|_{L^\infty(B_1)}\leq  C \exp\left(-\frac{3\,{\rm dist}(x_0,\Gamma(\overrightarrow{\boldsymbol{\gamma}}))}{32\varepsilon}\right)\quad\forall x\in B_{1/2}\,.
\end{equation}
By linearity of the equation, the gradient vector  $\nabla w_\varepsilon$  satisfies $-\Delta (\nabla w_\varepsilon)=1/4\nabla w_\varepsilon$ in $B_1$. Applying again \cite[Theorem 3.9]{GiTr} to each component of $\nabla w_\varepsilon$ in the smaller ball $B_{1/2}$, we deduce from \eqref{gradbd} that 
$$|\nabla^2 w_\varepsilon(x)|\leq C\|\nabla w_\varepsilon\|_{L^\infty(B_{1/2})}\leq  C \exp\left(-\frac{3\,{\rm dist}(x_0,\Gamma(\overrightarrow{\boldsymbol{\gamma}}))}{32\varepsilon}\right)\quad\forall x\in B_{1/4}\,.$$
Noticing that $|\nabla w_\varepsilon(0)|=\varepsilon|\nabla u(x_0)|$ and $|\nabla^2 w_\varepsilon(0)|=\varepsilon^2|\nabla^2 u(x_0)|$,  the conclusion follows. 
\vskip3pt

\noindent{\it Step 2 (Boundary estimates).} Let $\Omega_1\subset \Omega$ be a smooth and convex open set such that $\overline\Omega_0\subset \Omega_1$ and  $\min\{{\rm dist}(z,\partial\Omega\cup\partial\Omega_0):z\in\partial\Omega_1\}\geq \boldsymbol{\eta}_0/4$. Consider the smooth open set $U:=\Omega\setminus\overline\Omega_1$, and the function $v:\overline U\to \mathbb{R}$ given by $v:=1-u_{\overrightarrow{\boldsymbol{\gamma}}}$. Then $v$ satisfies $-\Delta v=(1/4\varepsilon^2)v$ in $U$, and $v=0$ on $\partial\Omega$.  On the other hand, Lemma \ref{comp} and Step 1 imply that 
$$\frac{1}{\varepsilon^2}\|v\|_{L^\infty(U)}+\|v\|_{C^{1,1}(\partial\Omega_1)}\leq C_{\boldsymbol{\eta}_0} \exp\left(-\frac{\boldsymbol{\eta}_0}{64\varepsilon}\right)\,.$$
From \cite[Theorem 8.33]{GiTr} we deduce that 
$$\frac{1}{\varepsilon^2}\|v\|_{C^{1}(U)}\leq C_{\boldsymbol{\eta}_0}  \exp\left(-\frac{\boldsymbol{\eta}_0}{128\varepsilon}\right)\,.$$
Setting $V_{\boldsymbol{\eta}_0}:=\{x\in\Omega: {\rm dist}(x,\partial\Omega)<\boldsymbol{\eta}_0/5\}$, \cite[Theorem 4.12]{GiTr} now implies  
$$\|v\|_{C^2(V_{\boldsymbol{\eta}_0})}\leq C_{\boldsymbol{\eta}_0} \exp\left(-\frac{\boldsymbol{\eta}_0}{128\varepsilon}\right)\,.$$
This last estimate leads to the conclusion since ${\rm dist}(x_0,\Gamma(\overrightarrow{\boldsymbol{\gamma}}))\geq \boldsymbol{\eta}_0/4$  for every $x_0\in V_{\boldsymbol{\eta}_0}$. 
\end{proof}

\begin{lemma}\label{estgradhessdist}
Let $\overrightarrow{\boldsymbol{\gamma}}\in\mathscr{P}_\Lambda(a_0,\mu)$. At every $x_0\in \overline\Omega\setminus \Gamma(\overrightarrow{\boldsymbol{\gamma}})$ satisfying ${\rm dist}(x_0,\Gamma(\overrightarrow{\boldsymbol{\gamma}}))\leq 13\varepsilon$, we have
$$\big|\nabla u_{\overrightarrow{\boldsymbol{\gamma}}}(x_0)\big| \leq \frac{C_{\boldsymbol{\eta}_0}}{{\rm dist}(x_0,\Gamma(\overrightarrow{\boldsymbol{\gamma}}))} \,,$$
and 
$$\big|\nabla^2 u_{\overrightarrow{\boldsymbol{\gamma}}}(x_0)\big| \leq \frac{C_{\boldsymbol{\eta}_0}}{{\rm dist}^2(x_0,\Gamma(\overrightarrow{\boldsymbol{\gamma}}))} \,,$$
for some constant $C_{\boldsymbol{\eta}_0}$ depending only on $\Omega$ and $\boldsymbol{\eta}_0$ (given in Lemma \ref{estigradhessloin}). 
\end{lemma}

\begin{proof}
By Lemma \ref{estigradhessloin}, we can assume that $\varepsilon<\boldsymbol{\eta}_0/26$. Then ${\rm dist}(x_0,\partial\Omega)>\boldsymbol{\eta}_0/2$, and  
setting $R:={\rm dist}(x_0,\Gamma(\overrightarrow{\boldsymbol{\gamma}}))\leq 13\varepsilon$, we have $B(x_0,R)\subset\Omega$.  

Since $-\Delta u_{\overrightarrow{\boldsymbol{\gamma}}}=1/(4\varepsilon^2) (1-u_{\overrightarrow{\boldsymbol{\gamma}}})$ in $B(x_0,R)$ and $0\leq u_{\overrightarrow{\boldsymbol{\gamma}}}\leq 1$, we deduce from \cite[Lemma~A.1]{BBH1}  that for $x\in B(x_0,R/2)$, 
$$|\nabla u_{\overrightarrow{\boldsymbol{\gamma}}}(x)|^2\leq C\left( \frac{\|1-u_{\overrightarrow{\boldsymbol{\gamma}}}\|_{L^\infty(B(x_0,R))}}{\varepsilon^2}+\frac{\|u_{\overrightarrow{\boldsymbol{\gamma}}}\|_{L^\infty(B(x_0,R))}}{(R-|x-x_0|)^2}\right)\|u_{\overrightarrow{\boldsymbol{\gamma}}}\|_{L^\infty(B(x_0,R))}
\leq \frac{C}{R^2} \,,$$
for some universal constant $C$. Now, the gradient vector field $\nabla u_{\overrightarrow{\boldsymbol{\gamma}}}$ satisfies the equation 
$$-\Delta (\nabla u_{\overrightarrow{\boldsymbol{\gamma}}})=-\frac{1}{4\varepsilon^2}\nabla u_{\overrightarrow{\boldsymbol{\gamma}}}\quad\text{in $B(x_0,R)$}\,,$$ 
and $\|\nabla u_{\overrightarrow{\boldsymbol{\gamma}}}\|_{L^\infty(B(x_0,R/2))}\leq CR^{-1}$. Applying again  \cite[Lemma A.1]{BBH1} in $B(x_0,R/2)$ to each component of $\nabla u_{\overrightarrow{\boldsymbol{\gamma}}}$ leads to 
$$ |\nabla^2 u_{\overrightarrow{\boldsymbol{\gamma}}}(x_0)|^2\leq C\left( \frac{1}{\varepsilon^2}+\frac{1}{R^2}\right)\|\nabla u_{\overrightarrow{\boldsymbol{\gamma}}}\|^2_{L^\infty(B(x_0,R/2))}\leq \frac{C}{R^4} \,,$$
and the proof is complete. 
\end{proof}

\begin{lemma}\label{lemgeom}
Let $\overrightarrow{\boldsymbol{\gamma}}\in\mathscr{P}(a_0,\mu)$. For every $\rho>0$, there exists a finite covering of $\Gamma(\overrightarrow{\boldsymbol{\gamma}})$ by closed balls $\{\overline B_j(x_j,\rho)\}_{j\in J}$ with $x_j\in \Gamma(\overrightarrow{\boldsymbol{\gamma}})$ such that 
$${\rm Card}(J)\leq \max\Big\{\min\big\{5\mathcal{H}^1(\Gamma(\overrightarrow{\boldsymbol{\gamma}}))\rho^{-1},25{\rm diam}(\Gamma(\overrightarrow{\boldsymbol{\gamma}}))^2\rho^{-2} \big\},1\Big\}\,.$$ 
In particular, 
$$ \mathcal{L}^2\Big(\big\{x\in\mathbb{R}^2:{\rm dist}(x,\Gamma(\overrightarrow{\boldsymbol{\gamma}}))\leq \rho\big\}\Big)\leq \max\Big\{20\pi\mathcal{H}^1(\Gamma(\overrightarrow{\boldsymbol{\gamma}}))\rho, 4\pi\rho^2\Big\}\,.$$
\end{lemma}

\begin{proof}
If $\rho\geq  {\rm diam}(\Gamma(\overrightarrow{\boldsymbol{\gamma}}))$, then we can cover $\Gamma(\overrightarrow{\boldsymbol{\gamma}})$ with the single ball $\overline B(a_0,\rho)$, and the announced estimates become trivial. Hence we can assume that $\rho<  {\rm diam}(\Gamma(\overrightarrow{\boldsymbol{\gamma}}))$. 
By compactness of $\Gamma(\overrightarrow{\boldsymbol{\gamma}})$, we can cover $\Gamma(\overrightarrow{\boldsymbol{\gamma}})$ with a finite collection of closed balls $\{\overline B(x_j,\rho/5)\}_{j\in \widetilde J}$  such that 
$x_j\in \Gamma(\overrightarrow{\boldsymbol{\gamma}})$.  
By the $5r$-covering theorem (see for instance \cite{Mattila}), we can find a subset $ J\subset \widetilde J$ such that $\overline B(x_i,\rho/5)\cap \overline B(x_j,\rho/5)=\emptyset$ if $i\not=j$ with $i,j\in J$, and  
$$\Gamma(\overrightarrow{\boldsymbol{\gamma}})\subset \bigcup_{j\in  J} \overline B(x_j,\rho)\,.$$
In particular,
$$\bigcup_{j\in  J} \overline B(x_j,\rho/5)\subset\Big\{x\in\mathbb{R}^2:{\rm dist}(x,\Gamma(\overrightarrow{\boldsymbol{\gamma}}))\leq \rho\Big\}\subset  \bigcup_{j\in  J} \overline B(x_j,2\rho)\,,$$
so that 
$$\frac{\pi}{25}\rho^2\,{\rm Card}(J) \leq  \mathcal{L}^2\Big(\big\{x\in\mathbb{R}^2:{\rm dist}(x,\Gamma(\overrightarrow{\boldsymbol{\gamma}}))\leq \rho\big\}\Big)\leq 4\pi\rho^2\,{\rm Card}(J)\,.$$
From the first inequality, we easily deduce that ${\rm Card}(J) \leq 25 {\rm diam}(\Gamma(\overrightarrow{\boldsymbol{\gamma}}))^2\rho^{-2} $. 

Next we claim that for each $j\in  J$, 
\begin{equation}\label{lwalf}
 \mathcal{H}^1(\Gamma(\overrightarrow{\boldsymbol{\gamma}})\cap B(x_j,\rho/5))\geq \rho/5\,. 
 \end{equation}
Note that this estimate leads to the announced result since 
$$\mathcal{H}^1(\Gamma(\overrightarrow{\boldsymbol{\gamma}}))\geq \sum_{j\in J}  \mathcal{H}^1(\Gamma(\overrightarrow{\boldsymbol{\gamma}})\cap B(x_j,\rho/5))\geq {\rm Card}(J)\rho/5\,. $$
To prove \eqref{lwalf}, we argue as follows. Since $\rho< {\rm diam}(\Gamma(\overrightarrow{\boldsymbol{\gamma}}))$, there exists a point $y_j\in \Gamma(\overrightarrow{\boldsymbol{\gamma}})\setminus B(x_j,\rho/5)$. On the other hand, the set $\Gamma(\overrightarrow{\boldsymbol{\gamma}})$ is arcwise connected since $\gamma_i(0)=a_0$ for each $i\in\{1,\ldots,N\}$. Hence, we can find a continuous path $\ell:[0,1]\to \Gamma(\overrightarrow{\boldsymbol{\gamma}})$ such that $\ell(0)=x_j$ and $\ell(1)=y_j$. Set 
$$t_*:=\sup\Big\{t: t\in[0,1]\text{ and } \ell(s)\in B(x_j,\rho/5)\text{ for every $s\in[0,t]$}  \Big\}\,. $$
By continuity of $\ell$, we have  $\ell(t_*)\in \partial B(x_j,\rho/5)$. Consequently, 
$$ \mathcal{H}^1(\Gamma(\overrightarrow{\boldsymbol{\gamma}})\cap B(x_j,\rho/5))\geq  \mathcal{H}^1\big(\ell([0,t_*))\big)\geq |\ell(t_*)-\ell(0)|=\rho/5\,,$$
which completes the proof.  
\end{proof}

We are now ready to prove the following higher integrability estimate, with explicit control with respect to the parameters. Here, the main point is the uniformity of the estimate with respect to $\mu/\|\mu\|$. The explicit dependence  with respect to $\varepsilon$ will be (strongly) used in the second part of our work \cite{BLM2}. 

\begin{proposition}\label{W1p}
If $\overrightarrow{\boldsymbol{\gamma}}\in \mathscr{P}_\Lambda(a_0,\mu)$, then $u_{\overrightarrow{\boldsymbol{\gamma}}}\in W^{1,p}(\Omega)$ for every $p<\infty$, and for $p>2$, 
\begin{multline*}
\|\nabla u_{\overrightarrow{\boldsymbol{\gamma}}}\|_{L^p(\Omega)}\\
\leq  C_{p,\boldsymbol{\eta}_0}\max\Big\{\!\min\big\{\mathcal{H}^{1}(\Gamma(\overrightarrow{\boldsymbol{\gamma}})), \frac{1}{\eps|\log\eps|} \big\},\varepsilon|\log\varepsilon|\Big\}^{1/p}\bigg(\frac{|\log\eps|^{1+1/p}}{\eps^{1-1/p}}\\+\frac{\Lambda\|\mu\||\log \varepsilon|^{1/p}}{\lambda_\varepsilon\varepsilon^{1-1/p}}\bigg)\,,
\end{multline*}
for some constant $C_{p,\boldsymbol{\eta}_0}$ depending only on $p$, $\Omega$, and  $\boldsymbol{\eta}_0$ (given in Lemma \ref{estigradhessloin}). 
\end{proposition}

\begin{proof}
{\it Step 1.} Replacing $\lambda_\eps$ by $\lambda_\eps/\|\mu\|$ and $\mu$ by $\mu/\|\mu\|$, we may assume that $\|\mu\|=1$. Without loss of generality, we can also assume that $\varepsilon|\log\varepsilon|<\boldsymbol{\eta}_0/256$.  
Let us fix some point $x_0\in\overline\Omega_0$ and $0<\rho< \boldsymbol{\eta}_0/4$.  
Let $T_\rho\in\mathscr{D}^\prime(\mathbb{R}^2)$ be the distribution defined by
$$\langle T_\rho,\varphi\rangle:= \sum_{i=1}^N\beta_i\int_{\Gamma(\gamma_i)}u_{\overrightarrow{\boldsymbol{\gamma}}}\varphi_\rho\,d\mathcal{H}^1=B_\mu[\overrightarrow{\boldsymbol{\gamma}}](u_{\overrightarrow{\boldsymbol{\gamma}}},\varphi_\rho)\,,$$
where $\varphi_\rho(x):=\varphi((x-x_0)/\rho)$ and $\varphi\in C^\infty_c(\mathbb{R}^2)$. 

 By Lemma \ref{bound1} and \eqref{estiZiem}, for every $\varphi\in C^\infty_c(B_2)$ we have 
$$\big|\big\langle T_\rho,\varphi\rangle\big| \leq\sum_{i=1}^N\beta_i\int_{\Gamma(\gamma_i)}|\varphi_\rho|\,d\mathcal{H}^1\leq C\Lambda\int_{B(x_0,2\rho)}|\nabla\varphi_\rho|\,dx
= C\Lambda \rho\int_{B(0,2)}|\nabla\varphi|\,dx \,.$$
Then we infer from H\"older's inequality that
$$\big|\big\langle T_\rho,\varphi\rangle\big| \leq C\Lambda \rho \|\nabla\varphi\|_{L^q(B_2)}\quad \forall \varphi\in C^\infty_c(B_2)\,,\forall \,1\leq q \leq 2\,.$$
Therefore $T_\rho\in W^{-1,p}(B_2)$ with 
$$\|T_\rho\|_{W^{-1,p}(B_2)}\leq   C\Lambda \rho$$
for every $2\leq p<\infty$.
\vskip3pt

\noindent{\it Step 2.} Let us now fix the exponent $2< p<\infty$.  By our choice of $\rho$, we have $B(x_0,2\rho)\subset\Omega$. 
As a consequence of Step 1, there exists a vector field $f\in L^p(B_2;\mathbb{R}^2)$ such that ${\rm div}\,f=T_\rho$ in $\mathscr{D}^\prime(B_2)$ and satisfying 
$$C_p^{-1}\|T_\rho\|_{W^{-1,p}(B_2)}\leq \|f\|_{L^p(B_2)}\leq C_p\|T_\rho\|_{W^{-1,p}(B_2)}$$
(see e.g. \cite[Sections 3.7 to 3.14]{AdF}). By classical elliptic theory (see e.g. \cite[Theorem~9.15 and Lemma 9.17]{GiTr}, there exists a (unique) solution $\xi\in W^{2,p}(B_2)\cap W^{1,p}_0(B_2)$ of 
$$\begin{cases}
-\Delta \xi=f &\text{in $B_2$}\,,\\
\xi=0 & \text{on $\partial B_2$}\,,
\end{cases}$$
satisfying the estimate
$$\|\xi\|_{W^{2,p}(B_2)}\leq C_p \|f\|_{L^p(B_2)}\leq C_p\Lambda\rho\,, $$
thanks to Step 1.

Now we define $v_\rho:={\rm div}\,\xi\in W^{1,p}(B_2)$ which satisfies 
$$-\Delta v_\rho=T_\rho \quad\text{in $\mathscr{D}^\prime(B_2)$}\,,$$
together with the estimate
$$\|v_\rho\|_{W^{1,p}(B_2)}\leq C_p \Lambda\rho\,. $$
Notice that, by the Sobolev embedding Theorem, $v_\rho\in L^\infty(B_2)$ and 
\begin{equation}\label{coucou1124}
\|v_\rho\|_{L^{\infty}(B_2)}\leq C_p\|v_\rho\|_{W^{1,p}(B_2)}\leq C_p\Lambda\rho\,.
\end{equation}
\vskip3pt

\noindent{\it Step 3.} Next we define for $x\in B_2$, $u_\rho(x):=u_{\overrightarrow{\boldsymbol{\gamma}}}(x_0+\rho x)$. Notice that  
$$-\Delta u_\rho=\frac{\rho^2}{4\varepsilon^2}(1-u_\rho) -\frac{1}{\lambda_\varepsilon\varepsilon}T_\rho\quad \text{in $\mathscr{D}^\prime(B_2)$}\,.$$
Indeed, for $\varphi\in C^\infty_c(B_2)$ we have 
\begin{align*}
\int_{B_2}\nabla u_\rho\nabla\varphi\,dx&=\int_{B(x_0,2\rho)}\nabla u_{\overrightarrow{\boldsymbol{\gamma}}}\nabla\varphi_\rho\,dx\\
&=\frac{1}{4\varepsilon^2}\int_{B(x_0,2\rho)}(1-u_{\overrightarrow{\boldsymbol{\gamma}}})\varphi_\rho\,dx-\frac{1}{\lambda_\eps\varepsilon}B[\overrightarrow{\boldsymbol{\gamma}}](u_{\overrightarrow{\boldsymbol{\gamma}}},\varphi_\rho)\\
&=\frac{\rho^2}{4\varepsilon^2}\int_{B_2}(1-u_\rho)\varphi\,dx- \frac{1}{\lambda_\eps\varepsilon}\langle T_\rho,\varphi\rangle\,.
\end{align*}
Consider the function 
$w_\rho:=u_\rho+\frac{1}{\lambda_\varepsilon\varepsilon}v_\rho \in H^1(B_2)\cap L^\infty(B_2)$  
which (therefore) satisfies
$$-\Delta w_\rho= \frac{\rho^2}{4\varepsilon^2}(1-u_\rho)\quad\text{in $B_2$} \,.$$
By \cite[Corollary 8.36]{GiTr}, $w_\rho\in C^{1,\alpha}_{\rm loc}(B_2)$ for some $\alpha>0$, and 
$$\|\nabla w_\rho\|_{L^\infty(B_1)}\leq C\left(\frac{\rho^2}{\varepsilon^2}\|1-u_\rho\|_{L^\infty(B_2)}+\|w_\rho\|_{L^\infty(B_2)}\right)\leq C_p\left(\frac{\rho^2}{\varepsilon^2}+1+\frac{\Lambda\rho}{\lambda_\varepsilon\varepsilon}\right)\,,$$
in view of \eqref{coucou1124} and the fact that $0\leq u_\rho\leq 1$.  Going back to $u_\rho=w_\rho-\frac{1}{\lambda_\varepsilon\varepsilon}v_\rho$, we deduce that $u_\rho\in W^{1,p}(B_1)$ with the estimate
\begin{equation}\label{estiurhoW1p}
\|\nabla u_\rho\|_{L^p(B_1)}\leq \|\nabla w_\rho\|_{L^\infty(B_1)}+\frac{\|\nabla v_\rho\|_{L^p(B_1)}}{\lambda_\varepsilon\varepsilon}\leq C_p\left(\frac{\rho^2}{\varepsilon^2}+1+\frac{\Lambda\rho}{\lambda_\varepsilon\varepsilon}\right)\,.
\end{equation}
Scaling back we finally obtain 
\begin{equation}\label{esti1103}
\|\nabla u_{\overrightarrow{\boldsymbol{\gamma}}}\|^p_{L^p(B(x_0,\rho))}\leq C_p\left(\frac{\rho^{p+2}}{\varepsilon^{2p}}+\frac{1}{\rho^{p-2}}+\frac{\Lambda^p\rho^{2}}{\lambda^p_\varepsilon\varepsilon^p}\right)\,.
\end{equation}
\vskip3pt

\noindent{\it Step 4.} Applying Lemma \ref{lemgeom}, we can cover $\Gamma(\overrightarrow{\boldsymbol{\gamma}})$ by finitely many balls $\{\overline B(x_j,\rho/2)\}_{j\in J}$ with $x_j\in \Gamma(\overrightarrow{\boldsymbol{\gamma}})$ and 
$$\rho\, {\rm Card}(J)\leq C\max\big\{\min\{\mathcal{H}^1(\Gamma(\overrightarrow{\boldsymbol{\gamma}})),\rho^{-1}\},\rho\big\}\,.$$ 
Then, 
$$V_{\rho/2}:=\{x\in\Omega:{\rm dist}(x,\Gamma(\overrightarrow{\boldsymbol{\gamma}}))<\rho/2\}\subset \bigcup_{j\in J} B(x_j,\rho)\,,$$ 
and we deduce from \eqref{esti1103} that 
\begin{multline*}
\int_{V_{\rho/2}} |\nabla u_{\overrightarrow{\boldsymbol{\gamma}}}|^p\,dx\leq \sum_{j\in J} \int_{B(x_j,\rho)} |\nabla u|^p\,dx\\
\leq C_p\max\Big\{\!\min\big\{\mathcal{H}^1(\Gamma(\overrightarrow{\boldsymbol{\gamma}})),\rho^{-1}  \big\},\rho\Big\}\left(\frac{\rho^{p+1}}{\varepsilon^{2p}}+\frac{1}{\rho^{p-3}}+\frac{\Lambda^p\rho}{\lambda^p_\varepsilon\varepsilon^p}\right)\,.
\end{multline*}
In particular, 
 \begin{multline}\label{estinum}
 \|\nabla u_{\overrightarrow{\boldsymbol{\gamma}}}\|_{L^p(V_{\rho/2})} \\
 \leq C_p\max\Big\{\!\min\{\mathcal{H}^{1}(\Gamma(\overrightarrow{\boldsymbol{\gamma}})),\rho^{-1}\},\rho\Big\}^{1/p}\left(\frac{\rho^{1+1/p}}{\varepsilon^{2}}+\frac{1}{\rho^{1-3/p}}+\frac{\Lambda\rho^{1/p}}{\lambda_\varepsilon\varepsilon}\right)\,.
 \end{multline}
Observe that, using the gradient estimate in Lemma  \ref{estigradhessloin}, the choice $\rho=64\varepsilon|\log\varepsilon|$ yields $\big|\nabla u_{\overrightarrow{\boldsymbol{\gamma}}}\big|\leq C_{\boldsymbol{\eta}_0}$ in $\Omega\setminus V_{32\varepsilon|\log\varepsilon|}$. Plugging this value of $\rho$ in  \eqref{estinum}, we deduce that
\begin{multline*}
\|\nabla u_{\overrightarrow{\boldsymbol{\gamma}}}\|_{L^p(V_{32\varepsilon|\log\varepsilon|})}\\ 
\leq C_p\max\Big\{\min\big\{\mathcal{H}^{1}(\Gamma(\overrightarrow{\boldsymbol{\gamma}})), \frac{1}{\eps|\log\eps|} \big\},\varepsilon|\log\varepsilon|\Big\}^{1/p}\bigg(\frac{|\log\eps|^{1+1/p}}{\eps^{1-1/p}}\\+\frac{\Lambda|\log \varepsilon|^{1/p}}{\lambda_\varepsilon\varepsilon^{1-1/p}}\bigg)\,,
\end{multline*}
and the conclusion follows.  
\end{proof}

\begin{proposition}\label{Holder}
If $\overrightarrow{\boldsymbol{\gamma}}\in \mathscr{P}_\Lambda(a_0,\mu)$, then  $u_{\overrightarrow{\boldsymbol{\gamma}}}\in C^{0,\alpha}(\Omega)$  for every $0<\alpha<1$, and 
$$\|u_{\overrightarrow{\boldsymbol{\gamma}}}\|_{C^{0,\alpha}(\Omega)}\leq C_{\alpha,\boldsymbol{\eta}_0}\frac{(1+\Lambda\|\mu\|\lambda^{-1}_\varepsilon)}{\varepsilon^\alpha}\,,$$
for some constant $C_{\alpha,\boldsymbol{\eta}_0}$ depending only on $\alpha$, $\Omega$, and $\boldsymbol{\eta}_0$ (given in Lemma \ref{estigradhessloin}). 
\end{proposition}

\begin{proof}
Note that it is enough to prove the announced estimate when $\varepsilon$ is small; thus we can assume that $13\varepsilon<\boldsymbol{\eta}_0/4$. 
Recall that, upon replacing $\lambda_\eps$ by $\lambda_\eps/\|\mu\|$ and $\mu$ by $\mu/\|\mu\|$, we can also assume that $\|\mu\|=1$.
Then we fix some distinct points $x,y\in\Omega$, and we set $x_0:=(x+y)/2$. 

If $|x-y|\geq \varepsilon$, then we have 
$$\frac{|u_{\overrightarrow{\boldsymbol{\gamma}}}(x)-u_{\overrightarrow{\boldsymbol{\gamma}}}(y)|}{|x-y|^\alpha}\leq \frac{2}{\varepsilon^\alpha} \,,$$
since $0\leq u_{\overrightarrow{\boldsymbol{\gamma}}}\leq 1$. 

Now we assume that $|x-y|< \varepsilon$. If ${\rm dist}(x_0,\partial\Omega)\leq \boldsymbol{\eta}_0/2$, then ${\rm dist}(z,\Gamma(\overrightarrow{\boldsymbol{\gamma}}))> \boldsymbol{\eta}_0/4$ for every $z\in B(x_0,\varepsilon)$, and the conclusion follows from Lemma \ref{estigradhessloin}. If ${\rm dist}(x_0,\partial\Omega)> \boldsymbol{\eta}_0/2$, then $B(x_0,\varepsilon)\subset\Omega$. Going back to estimate \eqref{estiurhoW1p} in the previous proof, we deduce that for $\rho=\varepsilon$ and $p=2/(1-\alpha)$,  
$$\|\nabla u_\varepsilon\|_{L^p(B_1)}\leq C_\alpha\left(1+ \frac{\Lambda}{\lambda_\varepsilon}\right)\,.$$
By the Sobolev embedding Theorem, the former estimate yields $\|u_\varepsilon\|_{C^{0,\alpha}(B(0,1))}\leq C_\alpha (1+\Lambda/\lambda_\varepsilon)$. 
Scaling back, we conclude that 
$$\frac{|u_{\overrightarrow{\boldsymbol{\gamma}}}(x)-u_{\overrightarrow{\boldsymbol{\gamma}}}(y)|}{|x-y|^\alpha}\leq C_\alpha \frac{(1+\Lambda\lambda^{-1}_\varepsilon)}{\varepsilon^\alpha}\,,$$
and the proof is complete. 
\end{proof}

\begin{remark}\label{counterex}
The regularity estimates in Proposition \ref{W1p} and Proposition \ref{Holder} are optimal in the sense that $\nabla u_{\overrightarrow{\boldsymbol{\gamma}}}\not\in L^\infty(\Omega)$ in general. To illustrate this fact, let us consider the simple case where $N=1$, $a_0=0$, $a_1=\tau$ for some $\tau\in\mathbb{S}^1$, and $\Gamma(\overrightarrow{\boldsymbol{\gamma}})=S:=[0,\tau]$ (the straight line segment). From the Euler-Lagrange equation \eqref{eqEL} and the continuity of $u_{\overrightarrow{\boldsymbol{\gamma}}}$, we have 
$$-\Delta u_{\overrightarrow{\boldsymbol{\gamma}}}= \frac{1}{4\varepsilon^2}(1-u_{\overrightarrow{\boldsymbol{\gamma}}}) -\frac{\beta_1}{\lambda_\varepsilon\varepsilon} u_{\overrightarrow{\boldsymbol{\gamma}}}\mathcal{H}^1\res S\quad\text{in $\mathscr{D}^\prime(\Omega)$}\,.$$
By elliptic regularity, $u_{\overrightarrow{\boldsymbol{\gamma}}}$ has essentially the regularity of the solution of the Poisson equation 
$$-\Delta v_*= -u_{\overrightarrow{\boldsymbol{\gamma}}}\mathcal{H}^1\res S\quad\text{in $\mathscr{D}^\prime(\mathbb{R}^2)$}\,,$$
given by the convolution of the measure $-u_{\overrightarrow{\boldsymbol{\gamma}}}\mathcal{H}^1\res S$ with the fundamental solution of the Laplacian, i.e., 
$$v_*(x):=\frac{1}{2\pi} \int_S \log(|x-y|)u_{\overrightarrow{\boldsymbol{\gamma}}}(y)\,d\mathcal{H}^1_y= \frac{1}{2\pi}\int_0^1\log( |x-t\tau|) u_{\overrightarrow{\boldsymbol{\gamma}}}(t\tau)\,dt\,.$$
Differentiating this formula, we obtain 
$$\nabla v_*(x)=\frac{1}{2\pi}\int_0^1\frac{(x-t\tau)}{|x-t\tau|^2}  u_{\overrightarrow{\boldsymbol{\gamma}}}(t\tau)\,dt\quad\text{for every $x\in\mathbb{R}^2\setminus S$}\,.$$
In particular, 
$$\tau\cdot \nabla v_*(s\tau)=\frac{1}{2\pi}\log\big(s/(1-s)\big)u_{\overrightarrow{\boldsymbol{\gamma}}}(s\tau)-\frac{1}{2\pi}\int_0^1\frac{u_{\overrightarrow{\boldsymbol{\gamma}}}(s\tau)-u_{\overrightarrow{\boldsymbol{\gamma}}}(t\tau)}{s-t}\,dt\quad\text{for $s>1$}\,.$$
In view of Proposition \ref{Holder}, we have for every $\alpha\in(0,1)$, 
$$|\nabla v_*(s\tau)|\geq \frac{1}{2\pi}|\log(s-1)| u_{\overrightarrow{\boldsymbol{\gamma}}}(s\tau)-C_\alpha\quad\text{for $s>1$}\,,$$
where $C_\alpha$ is a constant independent of $s$. Therefore $|\nabla v_*|$ cannot be essentially bounded near the point $\tau$ whenever $u_{\overrightarrow{\boldsymbol{\gamma}}}(\tau)\not=0$. Similarly,  $|\nabla v_*|$ is not bounded near $0$ whenever $u_{\overrightarrow{\boldsymbol{\gamma}}}(0)\not=0$. 
These last conditions are ensured for $\beta_1<\!<1$. Indeed, using Proposition \ref{Holder}, one may easily check that $u_{\overrightarrow{\boldsymbol{\gamma}}}\to 1$ uniformly in $\Omega$ as $\beta_1\to 0$ (with $\varepsilon$ fixed). 
\end{remark}


\subsection{Existence and regularity of minimizing pairs}

In this subsection, we move on the existence problem for minimizing pairs of the functional $E^\mu_\varepsilon$. Regularity of minimizers will essentially follow from our considerations about the problem with prescribed curves.  In all our statements, we shall use the {\it upper Alhfors threshold}
\begin{equation}\label{Lambdaeps}
\Lambda_\varepsilon:=2+\frac{3}{\delta_\varepsilon}\,.
\end{equation}
Our main results are the following. 

\begin{theorem}\label{existminpairs}
Assume that $\mu$ is of  the form \eqref{discrmeas}. The functional $E^\mu_\varepsilon$ admits at least one minimizing pair $(u_\varepsilon,\overrightarrow{\boldsymbol{\gamma}}_\varepsilon)$ in $(1+H^1_0(\Omega))\times \mathscr{P}(a_0,\mu)$. In addition, for any such minimizer, $\overrightarrow{\boldsymbol{\gamma}}_\varepsilon$ belongs to 
$ \mathscr{P}_{\Lambda_\varepsilon}(a_0,\mu)$, and $u_\varepsilon$ is the potential of $\overrightarrow{\boldsymbol{\gamma}}_\varepsilon$.
\end{theorem}

A byproduct of this theorem is the following existence and regularity result for our original functional $F^\mu_\varepsilon$ in case of a measure $\mu$ with finite support.

\begin{corollary}\label{thmminF}
Assume that $\mu$ is of  the form \eqref{discrmeas}. The functional $F^\mu_\varepsilon$ admits at least one minimizer  $u_\varepsilon$ in $1+H^1_0(\Omega)\cap L^\infty(\Omega)$. In addition, any such minimizer belongs to $W^{1,p}(\Omega)$ for every $p<\infty$ (in particular, $u_\varepsilon\in C^{0,\alpha}(\Omega)$ for every $\alpha\in(0,1)$). Moreover, there exists   $\overrightarrow{\boldsymbol{\gamma}}_\varepsilon\in \mathscr{P}(a_0,\mu)$ such that $(u_\varepsilon,\overrightarrow{\boldsymbol{\gamma}}_\varepsilon)$ is a minimizing pair  of $E^\mu_\varepsilon$ in $(1+H^1_0(\Omega))\times \mathscr{P}(a_0,\mu)$.
\end{corollary}

In the same way, we have an analogous result concerning the auxiliary functional $G^\mu_\varepsilon$ defined in \eqref{defGmu}. 

\begin{corollary}\label{thmminG}
Assume that $\mu$ is of  the form \eqref{discrmeas}. The functional $G^\mu_\varepsilon$ admits at least one minimizer $\overrightarrow{\boldsymbol{\gamma}}_\varepsilon=(\gamma_1^\varepsilon,\ldots,\gamma_N^\varepsilon)\in \mathscr{P}(a_0,\mu)$. In addition, any such minimizer belongs to $ \mathscr{P}_{\Lambda_\varepsilon}(a_0,\mu)$,  and $(u_{\overrightarrow{\boldsymbol{\gamma}}_\varepsilon}, \overrightarrow{\boldsymbol{\gamma}}_\varepsilon)$ is a minimizing pair of  $E^\mu_\varepsilon$ in $(1+H^1_0(\Omega))\times \mathscr{P}(a_0,\mu)$.
\end{corollary}

\begin{remark}\label{regC1alphaCurve0} Concerning the regularity of $\Gamma( \overrightarrow{\boldsymbol{\gamma}}_\varepsilon)$, we can invoke the results of \cite{DPLM} and the H\"older continuity of $u_\varepsilon$ to show that each $\Gamma(\gamma_i^\varepsilon)$ is in fact a $C^{1,\alpha}$ curve for every $\alpha\in(0,1/2)$ in a neighborhood of every point in $\Omega\setminus\{a_0,\ldots,a_N\}$ (assuming eventually that $\partial\Omega_0$ is smooth). One could use this further information to get improved (partial) regularity on $u_\varepsilon$, but we do not pursue this issue here. We also believe that the curves admit a tangent line at the $a_i$'s, and that the  $C^{1,\alpha}$ regularity holds true up to each $a_i$. This latter fact does not derive directly from the statements of  \cite{DPLM}, but can certainly be proved using the material developed there.  
\end{remark}

\begin{remark}\label{regC1alphaCurve} 
In all the statements above, we believe the regularity of $u_\varepsilon$ to be optimal in the sense that $u_\varepsilon$ is {\it not} Lipschitz continuous. More precisely, Lipschitz continuity should fail near the $a_i$'s. In view of Remarks \ref{counterex} \& \ref{regC1alphaCurve0}, the question boils down to determine whether or not $u_\varepsilon(a_i)$ vanishes or not. Up to some trivial situations, we believe that  $u_\varepsilon(a_i)\not=0$, and that $|\nabla u_\veps|$ actually behaves like $|\log(|x-a_i|)|$ in the neighborhood of  $a_i$ (as in Remark \ref{regC1alphaCurve0}).  
\end{remark}

Theorem \ref{existminpairs}, Corollary \ref{thmminF}, and Corollary \ref{thmminG} follow from the regularity estimates obtained in the previous subsection together with a set of lemmas of independent interest. Our first fundamental step is a replacement procedure allowing to show the upper Alhfors regularity of the curves. 

 \begin{lemma}\label{replace}
Let $u\in 1+H^1_0(\Omega)\cap L^\infty(\Omega)$ be such that $\|u\|_{L^\infty(\Omega)}\leq 1$, and let $\overrightarrow{\boldsymbol{\gamma}}=(\gamma_1,\ldots,\gamma_N)\in \mathscr{P}(a_0,\mu)$. If for some $i_0\in\{1,\ldots,N\}$,  
 $x\in \Gamma(\gamma_{i_0})$, and  $r>0$,  
\begin{equation}\label{contrAlf}
\mathcal{H}^1\big(\Gamma(\gamma_{i_0})\cap B(x,r)\big)\geq \Lambda_\varepsilon r \,,
\end{equation}
where $\Lambda_\veps$ is defined in  \eqref{Lambdaeps}, then there exists $\overrightarrow{\boldsymbol{\gamma}}_\sharp=(\gamma_1,\ldots,\gamma_{i_0-1},\gamma^\sharp_{i_0},\gamma_{i_0+1},\ldots,\gamma_N)\in \mathscr{P}(a_0,\mu)$ such that 
$$E^\mu_\eps(u,\overrightarrow{\boldsymbol{\gamma}}_\sharp)\leq E^\mu_\eps(u,\overrightarrow{\boldsymbol{\gamma}}) -\frac{\beta_{i_0} r}{\lambda_\eps}\,. $$
\end{lemma}

\begin{proof}
Assume that \eqref{contrAlf} holds. We shall suitably modify $\Gamma(\gamma_{i_0})$ in $B(x,r)$ to produce the competitor $\overrightarrow{\boldsymbol{\gamma}}_\sharp$. We proceed as follows. We first define 
$$t_{\rm in}:=\begin{cases}
\sup\big\{t\in [0,1]:  \gamma_{i_0}(s)\not\in B(x,r) \text{ for all } s\in[0,t)\big\}& \text{if $a_0\not\in B(x,r)$}\,,\\
0 &\text{otherwise}\,,
\end{cases}
$$
and
$$t_{\rm out}:=\begin{cases}
\inf\big\{t\in [0,1]:  \gamma_{i_0}(s)\not\in B(x,r) \text{ for all } s\in(t,1]\big\}& \text{if $a_{i_0}\not\in B(x,r)$}\,,\\
1 &\text{otherwise}\,.
\end{cases}
$$
Then we set $a:=\gamma_{i_0}(t_{\rm in})$ and $b:=\gamma_{i_0}(t_{\rm out})$. We finally define
$$\gamma^\sharp_{i_0}(t):=\begin{cases}
\gamma_{i_0}(t) & \text{if $t\in[0,t_{\rm in}]\cup[t_{\rm out},1]$}\,,\\[5pt]
\displaystyle \frac{t-t_{\rm in}}{t_{\rm out}-t_{\rm in}}b+\frac{t_{\rm out}-t}{t_{\rm out}-t_{\rm in}}a & \text{if $t\in [t_{\rm in},t_{\rm out}]$}\,.
\end{cases}$$
Since $\overline\Omega_0$ is convex, we have 
$$\Gamma(\gamma_{i_0}^\sharp)\subset \big(\Gamma(\gamma_{i_0})\setminus B(x,r)\big)\cup[a,b]\subset\overline\Omega_0\,.$$
Now we estimate 
$$\frac{\beta_{i_0}}{\lambda_\eps}\int_{\Gamma(\gamma^\sharp_{i_0})\cap B(x,r)}\delta_\eps+u^2\,d\mathcal{H}^1\leq \frac{2\beta_{i_0}}{\lambda_\eps}(1+\delta_\eps)r \,,$$
and 
$$\frac{\beta_{i_0}}{\lambda_\eps}\int_{\Gamma(\gamma_{i_0})\cap B(x,r)}\delta_\eps+u^2\,d\mathcal{H}^1\geq\frac{\beta_{i_0}\delta_\eps}{\lambda_\eps}\mathcal{H}^1\big(\Gamma(\gamma_{i_0})\cap B(x,r)\big)\geq \frac{\beta_{i_0}}{\lambda_\eps}(3+2\delta_\eps) r \,.$$
Since $\big(\Gamma(\gamma^\sharp_{i_0})\setminus B(x,r)\big)\subset \big(\Gamma(\gamma_{i_0})\setminus B(x,r)\big)$, we conclude that
$$E^\mu_\eps(u,\overrightarrow{\boldsymbol{\gamma}}) - E^\mu_\eps(u,\overrightarrow{\boldsymbol{\gamma}}_\sharp)\geq \frac{\beta_{i_0}}{\lambda_\eps}(3+2\delta_\eps) r - \frac{2\beta_{i_0}}{\lambda_\eps}(1+\delta_\eps)r=\frac{\beta_{i_0}r}{\lambda_\eps}\,,$$
and the proof is complete. 
\end{proof}

The following lemma provides the existence of a minimizer $\overrightarrow{\boldsymbol{\gamma}}_\sharp$  in $\mathscr{P}_{\Lambda_\veps}(a_0,\mu)$ associated to some fixed smooth function $u$.

\begin{lemma}\label{geodlem}
 Let $u\in 1+H^1_0(\Omega)\cap C^1(\overline\Omega)$ be such that $0\leq u\leq 1$.  There exists $\overrightarrow{\boldsymbol{\gamma}}_\sharp=(\gamma_1^\sharp,\ldots,\gamma_N^\sharp)\in \mathscr{P}_{\Lambda_\veps}(a_0,\mu)$ satisfying 
\begin{equation}\label{geodprop}
E^\mu_\varepsilon(u,\overrightarrow{\boldsymbol{\gamma}}_\sharp)\leq  E^\mu_\varepsilon(u,\overrightarrow{\boldsymbol{\gamma}})\qquad \forall \overrightarrow{\boldsymbol{\gamma}}\in \mathscr{P}(a_0,\mu)\,,
\end{equation}
and such that each $\gamma_i^\sharp:[0,1]\to\overline\Omega_0$ is injective if $a_i\not=a_0$, and constant if $a_i=a_0$. 
\end{lemma}

\begin{proof}
If $a_i=a_0$, we choose $\gamma_i^\sharp$ to be the constant map equal to $a_i$. Then,  for each  $a_i\not=a_0$, we consider the minimization problem
$$\min_{\mathscr{P}(a_0,a_i)} \int_0^1\big(\delta_\varepsilon+u^2(\gamma(t))\big)|\gamma^\prime(t)|\,dt\,. $$
By \cite[Theorem 5.22]{BGH} this problem admits a solution $\gamma_i^\sharp$ satisfying 
$$\big(\delta_\varepsilon+u^2(\gamma_i^\sharp(t))\big)|(\gamma_i^\sharp)^\prime(t)|=h_i\quad\text{a.e. in $(0,1)$}\,, $$
for some constant $h_i>0$. We claim that $\gamma_i^\sharp$ is injective. Indeed, if $\gamma_i^\sharp(t_1)=\gamma_i^\sharp(t_2)$ for some $t_1<t_2$, then we can consider the competitor $\widetilde \gamma_i^\sharp\in\mathscr{P}(a_0,a_i)$ defined by 
$$\widetilde \gamma_i^\sharp(t)\begin{cases} 
\gamma_i^\sharp(t) & \text{for $t\in[0,t_1]$}\,,\\
\gamma_i^\sharp(t_1) & \text{for $t\in[t_1,t_2]$}\,,\\
\gamma_i^\sharp(t_1) & \text{for $t\in[t_2,1]$}\,.
\end{cases}$$
Comparing energies, we have 
$$ \int_0^1\big(\delta_\varepsilon+u^2(\widetilde \gamma_i^\sharp(t))\big)|(\widetilde \gamma_i^\sharp)^\prime(t)|\,dt-\int_0^1\big(\delta_\varepsilon+u^2(\gamma_i^\sharp(t))\big)|( \gamma_i^\sharp)^\prime(t)|\,dt=-h_i(t_2-t_2)<0\,,$$
which contradicts the minimality of $\gamma_i^\sharp$. 

Now we set $\overrightarrow{\boldsymbol{\gamma}}_\sharp=(\gamma_1^\sharp,\ldots,\gamma_N^\sharp)$, and we claim that  \eqref{geodprop} holds. Clearly, it is enough to show that for each $i\in\{1,\ldots,N\}$, 
\begin{equation}\label{intrgeodprop}
\int_{\Gamma(\gamma_i^\sharp)}(\delta_\varepsilon+u^2)\,d\mathcal{H}^1\leq  \int_{\Gamma(\gamma)}(\delta_\varepsilon+u^2)\,d\mathcal{H}^1\qquad\forall \gamma\in\mathscr{P}(a_0,a_i)\,.
\end{equation}
Obviously, this inequality holds if $a_i=a_0$ since the left hand side vanishes. Hence we may assume that $a_i\not=a_0$. Let us then consider an arbitrary $\gamma\in\mathscr{P}(a_0,a_i)$. Since $\mathcal{H}^1(\Gamma(\gamma))<\infty$, \cite[Theorem~4.4.7]{AT} tells us that there exists an injective curve $\widetilde \gamma\in \mathscr{P}(a_0,a_i)$ such that 
$\Gamma(\widetilde\gamma)\subset\Gamma(\gamma)$.  Now we infer from the area formula (see e.g. \cite[Theorem 2.71]{AFP}) and the minimality of $\gamma_i^\sharp$ that 
\begin{multline*}
\int_{\Gamma(\gamma_i^\sharp)}(\delta_\varepsilon+u^2)\,d\mathcal{H}^1=\int_0^1\big(\delta_\varepsilon+u^2(\gamma_i^\sharp(t))\big)|( \gamma_i^\sharp)^\prime(t)|\,dt\\
\leq \int_0^1\big(\delta_\varepsilon+u^2(\widetilde \gamma(t))\big)|\widetilde\gamma^\prime(t)|\,dt= \int_{\Gamma(\widetilde\gamma)}(\delta_\varepsilon+u^2)\,d\mathcal{H}^1\leq  \int_{\Gamma(\gamma)}(\delta_\varepsilon+u^2)\,d\mathcal{H}^1\,,
\end{multline*}
and \eqref{intrgeodprop} is proved. 

Finally, we notice that $\overrightarrow{\boldsymbol{\gamma}}_\sharp\in\mathscr{P}_{\Lambda_\veps}(a_0,\mu)$ as a direct consequence of \eqref{geodprop} and Lemma~\ref{replace}, and the proof is complete. 
\end{proof}

The next lemma will allow us to replace an arbitrary pair $(u,\overrightarrow{\boldsymbol{\gamma}})$ by a regular one, with controlled energy.

\begin{lemma}\label{goodpairs}
For every $\sigma>0$, $u\in 1+H^1_0(\Omega)$, and $\overrightarrow{\boldsymbol{\gamma}}\in\mathscr{P}(a_0,\mu)$, there exist  
$u_\sigma\in 1+H^1_0(\Omega)\cap C^1(\overline\Omega) $ and $\overrightarrow{\boldsymbol{\gamma}}_\sigma\in\mathscr{P}_{\Lambda_\eps}(a_0,\mu)$ such that $0\leq  u_\sigma\leq 1$ and 
$$E^\mu_\eps(u_\sigma,\overrightarrow{\boldsymbol{\gamma}}_\sigma)\leq E^\mu_\eps(u,\overrightarrow{\boldsymbol{\gamma}})+\sigma\,. $$
\end{lemma} 
 
 \begin{proof}
 We first claim that there exists $\widetilde u\in 1+H^1_0(\Omega)\cap C^0(\overline\Omega) $ such that $0\leq \widetilde u\leq 1$ and 
 $$E^\mu_\eps(\widetilde u,\overrightarrow{\boldsymbol{\gamma}})\leq E^\mu_\eps(u,\overrightarrow{\boldsymbol{\gamma}})+\sigma\,. $$
 Without loss of generality, we may assume that $E^\mu_\eps(u,\overrightarrow{\boldsymbol{\gamma}})<\infty$.
Moreover, by the truncation argument in the proof of Lemma \ref{bound1}, we can reduce the question to the case $0\leq u \leq 1$. Then write $u=1-v$ with $v\in H^1_0(\Omega)$. Since $C_c^\infty(\Omega)$ is dense in $H^1_0(\Omega)$, we can find a sequence $(v_n)_{n\in\mathbb{N}}\subset C_c^\infty(\Omega)$ such that $v_n\to v$ strongly in $H^1_0(\Omega)$ as $n\to\infty$. Since $0\leq v\leq 1$, we may even assume that $0\leq v_n\leq 1$. By \cite[Theorem~4.1.2]{BB} we can find a (not relabeled) subsequence such that $v_n\to v$ quasi-everywhere in $\Omega$ (i.e., $v_n\to v$ in the pointwise sense away from a set of vanishing $H^1$-capacity). Since a set of vanishing $H^1$-capacity is $\mathcal{H}^1$-null, we deduce that $v_n\to v$ $\mathcal{H}^1$-a.e. on $\Gamma(\overrightarrow{\boldsymbol{\gamma}})$. Then, by the dominated convergence, we have for each $i\in\{1,\ldots,N\}$, 
 $$\int_{\Gamma(\gamma_i)} \delta_\eps+(1-v_n)^2\,d\mathcal{H}^1\to \int_{\Gamma(\gamma_i)} \delta_\eps+(1-v)^2\,d\mathcal{H}^1\,.$$
 Setting $u_n:=1-v_n$, we conclude that for $n$ large enough, $E^\mu_\eps(u_n,\overrightarrow{\boldsymbol{\gamma}})\leq E^\mu_\eps(u,\overrightarrow{\boldsymbol{\gamma}})+\sigma$, and the claim is proved. 

Finally, we apply Lemma \ref{geodlem} to find $\overrightarrow{\boldsymbol{\gamma}}_\sharp\in\mathscr{P}_{\Lambda_\eps}(a_0,\mu)$ such that 
$$E^\mu_\eps(u_n,\overrightarrow{\boldsymbol{\gamma}}_\sharp)\leq E^\mu_\eps(u_n,\overrightarrow{\boldsymbol{\gamma}})\leq E^\mu_\eps(u,\overrightarrow{\boldsymbol{\gamma}})+\sigma\,,$$
and the announced result is proved for $u_\sigma:=u_n$ and  $\overrightarrow{\boldsymbol{\gamma}}_\sigma:=\overrightarrow{\boldsymbol{\gamma}}_\sharp$. 
\end{proof}
 
\begin{proof}[Proof of Theorem \ref{existminpairs}] 
{\it Step 1 (existence).} Let $\{(u_n,\overrightarrow{\boldsymbol{\gamma}}_n)\}_{n\in\mathbb{N}}$ be a minimizing sequence for $E_\varepsilon$ over $(1+H_0^1(\Omega))\times\mathscr{P}(a_0,\mu)$, i.e., 
 $$\lim_{n\to\infty}E^\mu_\varepsilon (u_n,\overrightarrow{\boldsymbol{\gamma}}_n)=\inf_{(1+H_0^1(\Omega))\times\mathscr{P}(a_0,\mu)}E^\mu_\varepsilon\,. $$
 By Lemma \ref{goodpairs}, there is no loss of generality assuming that $(u_n,\overrightarrow{\boldsymbol{\gamma}}_n)\in C^1(\overline\Omega)\times\mathscr{P}_{\Lambda_\varepsilon}(a_0,\mu)$ and $0\leq u_n\leq 1$. In addition, by Lemma \ref{geodlem} we can even assume that, setting $\overrightarrow{\boldsymbol{\gamma}}_n=(\gamma_1^n,\ldots,\gamma_N^n)$, all $\gamma_i^n$'s are injective curves for $a_i\not=a_0$, and constant for $a_i=a_0$. Then we consider the sequence $\{(u_{\overrightarrow{\boldsymbol{\gamma}}_n},\overrightarrow{\boldsymbol{\gamma}}_n)\}_{n\in\mathbb{N}}$, where $u_{\overrightarrow{\boldsymbol{\gamma}}_n}$ is the potential of $\overrightarrow{\boldsymbol{\gamma}}_n$, i.e., the minimizer of $E^\mu_\varepsilon(\cdot, \overrightarrow{\boldsymbol{\gamma}}_n)$ over $1+H^1_0(\Omega)$. Obviously, $\{(u_{\overrightarrow{\boldsymbol{\gamma}}_n},\overrightarrow{\boldsymbol{\gamma}}_n)\}_{n\in\mathbb{N}}$ is still a minimizing sequence by minimality of $u_{\overrightarrow{\boldsymbol{\gamma}}_n}$. 
 
 By Proposition \ref{Holder}, 
 $$\|u_{\overrightarrow{\boldsymbol{\gamma}}_n}\|_{C^{0,\alpha}(\Omega)}\leq C_{\alpha,\boldsymbol{\eta}_0}(\varepsilon)\qquad\forall \alpha\in(0,1)\,, $$
for some constant $C_{\alpha,\boldsymbol{\eta}_0}(\varepsilon)$ independent of $n$. By the Arzel\`a-Ascoli Theorem, we can extract a (not relabeled) subsequence such that $u_{\overrightarrow{\boldsymbol{\gamma}}_n}\to u_\varepsilon$ uniformly in $\Omega$ and weakly in $H^1(\Omega)$ for some function $u_\varepsilon\in 1+H^1_0(\Omega)\cap C^{0,\alpha}(\Omega)$ for every $\alpha\in (0,1)$. 

On the other hand, the energy being invariant under reparametrization, we can assume that each $\gamma_i^n$ is a constant speed parametrization of its image $\Gamma(\gamma_i^n)$. In particular,  each $\gamma_i^n$ is a $\mathcal{H}^1(\Gamma(\gamma_i^n))$-Lipschitz curve. Since 
$$\mathcal{H}^1(\Gamma(\gamma_i^n))\leq \frac{\lambda_\varepsilon}{\delta_\varepsilon} E^\mu_\varepsilon(u_{\overrightarrow{\boldsymbol{\gamma}}_n}, \overrightarrow{\boldsymbol{\gamma}}_n)\leq C(\varepsilon)\,,$$
we infer that each sequence $\{\gamma_i^n\}_{n\in\mathbb{N}}$ is equi-Lipschitz. Therefore, we can extract a further subsequence  such that, for each $i\in\{1,\ldots,N\}$, $\gamma_i^n\to\gamma_i^\varepsilon$ uniformly on $[0,1]$ and weakly* in $W^{1,\infty}(0,1)$ for some $\gamma_i^\varepsilon\in \mathscr{P}(a_0,a_i)$. Then we set $\overrightarrow{\boldsymbol{\gamma}}_\varepsilon:=(\gamma_1^\varepsilon,\ldots,\gamma_N^\varepsilon)\in\mathscr{P}(a_0,\mu)$. 

Let us now fix an arbitrary $\kappa\in(0,\delta_\varepsilon/2)$. By the uniform convergence of $u_{\overrightarrow{\boldsymbol{\gamma}}_n}$ towards $u_\varepsilon$, we have $u_\varepsilon^2\leq u^2_{\overrightarrow{\boldsymbol{\gamma}}_n}+\kappa$ in $\Omega$ for $n$ large enough. From the injectivity of each $\gamma_i^n$ (for $a_i\not=a_0$) and the area formula, we derive that for $a_i\not=a_0$ and $n$ large,
\begin{multline}\label{1546}
\int_{\Gamma(\gamma_i^n)}(\delta_\varepsilon+u^2_{\overrightarrow{\boldsymbol{\gamma}}_n})\,d\mathcal{H}^1\geq \int_{\Gamma(\gamma_i^n)}(\delta_\varepsilon-\kappa+u^2_\varepsilon)\,d\mathcal{H}^1\\
=\int_0^1\big(\delta_\varepsilon-\kappa+u^2_\varepsilon(\gamma_i^n(t))\big)|(\gamma_i^n)^\prime(t)|\,dt\,.
\end{multline}
Since $\gamma_i^n\mathop{\rightharpoonup}\limits^*\gamma_i^\varepsilon$ weakly* in $W^{1,\infty}((0,1))$, the lower semicontinuity result in \cite[Theorem~3.8]{MaSb} tells us that 
\begin{equation}\label{1551}
\liminf_{n\to\infty} \int_0^1\big(\delta_\varepsilon-\kappa+u^2_\varepsilon(\gamma_i^n(t))\big)|(\gamma_i^n)^\prime(t)|\,dt\geq \int_0^1\big(\delta_\varepsilon-\kappa+u^2_\varepsilon(\gamma_i^\varepsilon(t))\big)|(\gamma_i^\varepsilon)^\prime(t)|\,dt\,.
\end{equation}
By the area formula again, 
\begin{equation}\label{1552}
 \int_0^1\big(\delta_\varepsilon-\kappa+u^2_\varepsilon(\gamma_i^\varepsilon(t))\big)|(\gamma_i^\varepsilon)^\prime(t)|\,dt\geq \int_{\Gamma(\gamma_i^\varepsilon)}(\delta_\varepsilon-\kappa+u^2_\varepsilon)\,d\mathcal{H}^1\,.
 \end{equation}
Gathering \eqref{1546}, \eqref{1551}, \eqref{1552}, and letting $\kappa\to 0$, we deduce that 
$$ \liminf_{n\to\infty}\int_{\Gamma(\gamma_i^n)}(\delta_\varepsilon+u^2_{\overrightarrow{\boldsymbol{\gamma}}_n})\,d\mathcal{H}^1\geq  \int_{\Gamma(\gamma_i^\varepsilon)}(\delta_\varepsilon+u^2_\varepsilon)\,d\mathcal{H}^1\qquad\forall i\in\{1,\ldots,N\}\,.$$
(Note that for $a_i=a_0$, this inequality is trivial since $\gamma_i^n$ is the constant map equal to $a_0$.)  
Since the diffuse part of the energy is clearly lower semicontinuous with respect to weak $H^1$-convergence, we conclude that 
$$E^\mu_\varepsilon(u_\varepsilon,\overrightarrow{\boldsymbol{\gamma}}_\varepsilon)\leq \lim_{n\to\infty} E^\mu_\varepsilon (u_n,\overrightarrow{\boldsymbol{\gamma}}_n)\,,$$
and thus $(u_\varepsilon,\overrightarrow{\boldsymbol{\gamma}}_\varepsilon)$ is a minimizer of $E^\mu_\varepsilon$. 
\vskip3pt

\noindent{\it Step 2 (regularity).} Now we consider an arbitrary minimizer $(u_\varepsilon,\overrightarrow{\boldsymbol{\gamma}}_\varepsilon)$ of $E^\mu_\varepsilon$ in $(1+H_0^1(\Omega))\times\mathscr{P}(a_0,\mu)$. Arguing as in the proof of Lemma \ref{bound1}, we obtain $0\leq u_\varepsilon\leq 1$ by minimality of $u_\varepsilon$ for $E^\mu_\varepsilon(\cdot,\overrightarrow{\boldsymbol{\gamma}}_\varepsilon)$. In turn, the minimality of $\overrightarrow{\boldsymbol{\gamma}}_\varepsilon$ for $E^\mu_\varepsilon(u_\varepsilon,\cdot)$ implies that $\overrightarrow{\boldsymbol{\gamma}}_\varepsilon\in\mathscr{P}_{\Lambda_\varepsilon}(a_0,\mu)$ by Lemma \ref{replace}. Now Theorem \ref{existuniqthmgammafix} shows that $u_\varepsilon$ is the potential of $\overrightarrow{\boldsymbol{\gamma}}_\varepsilon$. 
\end{proof} 
 
 \begin{proof}[Proof of Corollary \ref{thmminF}]
 Existence of a minimizer of $F^\mu_\varepsilon$ in $1+H^1_0(\Omega)\cap L^\infty(\Omega)$ is ensured by Theorem \ref{existminpairs} since $\inf F^\mu_\varepsilon=\min E^\mu_\varepsilon$ by \eqref{relEF}. Let us now consider an arbitrary minimizer $u_\varepsilon$ of $F^\mu_\varepsilon$ in $1+H^1_0(\Omega)\cap L^\infty(\Omega)$. We first claim that $0\leq u_\varepsilon\leq 1$ a.e. in $\Omega$. Indeed, setting $v:=\max(\min(u_\varepsilon,1),0)\in 1+H^1_0(\Omega)$, we can argue as in the proof of Lemma \ref{bound1} to show $E^\mu_\varepsilon(v,\overrightarrow{\boldsymbol{\gamma}})\leq E^\mu_\varepsilon(u_\varepsilon,\overrightarrow{\boldsymbol{\gamma}})$ for every $\overrightarrow{\boldsymbol{\gamma}}\in\mathscr{P}(a_0,\mu)$. Hence $F^\mu_\varepsilon(v)\leq F^\mu_\varepsilon(u_\varepsilon)$ by  \eqref{relEF}, the inequality being strict whenever $\{v\not=u_\varepsilon\}$ has a non vanishing Lebesgue measure. The minimality of $u_\varepsilon$ then implies that $v=u_\varepsilon$ a.e. in $\Omega$. 
  
Next, by  definition of $F^\mu_\varepsilon$, there exists a sequence $\{\overrightarrow{\boldsymbol{\gamma}}_n\}_{n\in\mathbb{N}}\subset  \mathscr{P}(a_0,\mu)$ such that
 $$E^\mu_\varepsilon(u_\varepsilon,\overrightarrow{\boldsymbol{\gamma}}_n)\leq F^\mu_{\varepsilon}(u_\varepsilon)+2^{-n-1}\qquad\forall n\in\mathbb{N} \,.$$
 On the other hand, we can argue as in the proof of Lemma \ref{goodpairs} to find, for each $n\in\mathbb{N}$, a function $u_n\in (1+H_0^1(\Omega))\cap C^{1}(\overline\Omega)$ such that $0\leq u_n\leq 1$ in $\Omega$, $\|u_n-u_\varepsilon\|_{H^1(\Omega)}\leq 2^{-n}$, and 
 $$E^\mu_\varepsilon(u_n,\overrightarrow{\boldsymbol{\gamma}}_n)\leq E^\mu_\varepsilon(u_\varepsilon,\overrightarrow{\boldsymbol{\gamma}}_n)+2^{-n-1}\leq F^\mu_{\varepsilon}(u_\varepsilon)+2^{-n}\,. $$
 Applying Lemma \ref{geodlem} to each $u_n$, we find (injective or constant) curves $\overrightarrow{\boldsymbol{\gamma}}_{\sharp,n}\in \mathscr{P}_{\Lambda_\varepsilon}(a_0,\mu)$ of constant speed such that 
 $$  E^\mu_\varepsilon(u_n,\overrightarrow{\boldsymbol{\gamma}}_{\sharp,n})\leq E^\mu_\varepsilon(u_n,\overrightarrow{\boldsymbol{\gamma}}_n)\leq F^\mu_\varepsilon(u_\varepsilon)+2^{-n} \,.$$
 Now we consider the potential $u_{\overrightarrow{\boldsymbol{\gamma}}_{\sharp,n}}$ of $\overrightarrow{\boldsymbol{\gamma}}_{\sharp,n}$. Then, 
 \begin{equation}\label{est1620}
 E^\mu_\varepsilon(u_{\overrightarrow{\boldsymbol{\gamma}}_{\sharp,n}},\overrightarrow{\boldsymbol{\gamma}}_{\sharp,n})\leq
  E^\mu_\varepsilon(u_n,\overrightarrow{\boldsymbol{\gamma}}_{\sharp,n})\leq F^\mu_\varepsilon(u_\varepsilon)+2^{-n} \,.
  \end{equation}
  Setting $w_n:=u_n-u_{\overrightarrow{\boldsymbol{\gamma}}_{\sharp,n}}\in H^1_0(\Omega)$, we infer from the equation \eqref{eqEL} satisfied by $u_{\overrightarrow{\boldsymbol{\gamma}}_{\sharp,n}}$ that 
 \begin{multline*}
 2^{-n}\geq E^\mu_\varepsilon(u_n,\overrightarrow{\boldsymbol{\gamma}}_{\sharp,n})-  E^\mu_\varepsilon(u_{\overrightarrow{\boldsymbol{\gamma}}_{\sharp,n}},\overrightarrow{\boldsymbol{\gamma}}_{\sharp,n})=\varepsilon\int_{\Omega}|\nabla w_n|^2\,dx+\frac{1}{4\varepsilon}\int_{\Omega}|w_n|^2\,dx\\
 +\frac{1}{\lambda_\varepsilon}B_\mu[\overrightarrow{\boldsymbol{\gamma}}_{\sharp,n}](w_n,w_n)\,.
 \end{multline*}
 Consequently, $\|w_n\|_{H^1(\Omega)}\leq C_\varepsilon2^{-n/2}$, so that $\|u_\varepsilon-u_{\overrightarrow{\boldsymbol{\gamma}}_{\sharp,n}}\|_{H^1(\Omega)}\leq C_\varepsilon2^{-n/2}$. On the other hand, the sequence $\{u_{\overrightarrow{\boldsymbol{\gamma}}_{\sharp,n}}\}$  remains bounded in $W^{1,p}(\Omega)$ for each $p<\infty$ by Proposition \ref{W1p}.  Since $u_{\overrightarrow{\boldsymbol{\gamma}}_{\sharp,n}}\to u_\varepsilon$ in $H^1(\Omega)$, we conclude that $u_\varepsilon\in W^{1,p}(\Omega)$ for each $p<\infty$. In particular, $u_\varepsilon\in C^{0,\alpha}(\Omega)$ for every $\alpha\in(0,1)$, and $u_{\overrightarrow{\boldsymbol{\gamma}}_{\sharp,n}}\to u_\varepsilon$ uniformly  in $\Omega$.

To conclude, we proceed as in the proof of Theorem  \ref{existminpairs}, Step 1: for a (not relabeled) subsequence, $\overrightarrow{\boldsymbol{\gamma}}_{\sharp,n}\mathop{\rightharpoonup}\limits^*\overrightarrow{\boldsymbol{\gamma}}_{\varepsilon}$ weakly* in $W^{1,\infty}(0,1)$ for some $\overrightarrow{\boldsymbol{\gamma}}_{\varepsilon}\in \mathscr{P}(a_0,\mu)$, and 
$$\liminf_{n\to\infty}E^\mu_\varepsilon(u_{\overrightarrow{\boldsymbol{\gamma}}_{\sharp,n}},\overrightarrow{\boldsymbol{\gamma}}_{\sharp,n})\geq E^\mu_\varepsilon(u_\varepsilon,\overrightarrow{\boldsymbol{\gamma}}_{\varepsilon})\geq F^\mu_\varepsilon(u_\varepsilon)\,.$$
 In view of \eqref{est1620}, we have $F^\mu_\varepsilon(u_\varepsilon)=E^\mu_\varepsilon(u_\varepsilon,\overrightarrow{\boldsymbol{\gamma}}_{\varepsilon})$, which shows that $(u_\varepsilon,\overrightarrow{\boldsymbol{\gamma}}_{\varepsilon})$ is a minimizer of $E^\mu_\varepsilon$ in $(1+H^1_0(\Omega))\times\mathscr{P}(a_0,\mu)$. 
\end{proof}
 
 \begin{proof}[Proof of Corollary \ref{thmminG}]
 Existence of a minimizer of $G^\mu_\varepsilon$ is ensured by Theorem~\ref{existminpairs} since $\inf G^\mu_\varepsilon=\min E^\mu_\varepsilon$. Let us now consider an arbitrary minimizer $\overrightarrow{\boldsymbol{\gamma}}_{\varepsilon}$ in  $\mathscr{P}(a_0,\mu)$. We first claim that $\overrightarrow{\boldsymbol{\gamma}}_{\varepsilon}=(\gamma_1^\varepsilon,\ldots,\gamma_N^\varepsilon)\in \mathscr{P}_{2\Lambda_\varepsilon}(a_0,\mu)$. Assume by contradiction that it does not belongs to $\mathscr{P}_{2\Lambda_\varepsilon}(a_0,\mu)$. Then we can find $i_0\in\{1,\ldots,N\}$, $x_0\in\Gamma(\gamma^\varepsilon_{i_0})$, and $r>0$ such that  
 $$\mathcal{H}^1(\Gamma(\gamma_{i_0}^\varepsilon)\cap B(x_0,r))\geq \Lambda_\varepsilon r\,. $$
 By the very definition of $G^\mu_\varepsilon$, we can find $\widetilde u\in 1+H^1_0(\Omega)$ such that 
$$E^\mu_\varepsilon(\widetilde u,\overrightarrow{\boldsymbol{\gamma}}_{\varepsilon})\leq G^\mu_\varepsilon(\overrightarrow{\boldsymbol{\gamma}}_{\varepsilon})+\frac{\beta_{i_0}r}{2\lambda_\varepsilon}\,.$$
Arguing as in the proof of Lemma \ref{bound1}, we may assume that $0\leq \widetilde u\leq  1$. Then, by Lemma~\ref{replace} there exists  $\overrightarrow{\boldsymbol{\gamma}}_{\sharp}\in \mathscr{P}(a_0,\mu)$ such that 
$$G^\mu_\varepsilon(\overrightarrow{\boldsymbol{\gamma}}_{\sharp})\leq E^\mu_\varepsilon(\widetilde u,\overrightarrow{\boldsymbol{\gamma}}_{\sharp})\leq  E^\mu_\varepsilon(\widetilde u,\overrightarrow{\boldsymbol{\gamma}}_{\varepsilon})-\frac{\beta_{i_0}r}{\lambda_\varepsilon}\leq G^\mu_\varepsilon(\overrightarrow{\boldsymbol{\gamma}}_{\varepsilon})-\frac{\beta_{i_0}r}{2\lambda_\varepsilon}<G^\mu_\varepsilon(\overrightarrow{\boldsymbol{\gamma}}_{\varepsilon})\,,$$
which contradicts the minimality of $\overrightarrow{\boldsymbol{\gamma}}_{\varepsilon}$. 

Since $\overrightarrow{\boldsymbol{\gamma}}_{\varepsilon}\in  \mathscr{P}_{2\Lambda_\varepsilon}(a_0,\mu)$, we conclude that $G^\mu(\overrightarrow{\boldsymbol{\gamma}}_{\varepsilon})=E^\mu_\varepsilon(u_{\overrightarrow{\boldsymbol{\gamma}}_{\varepsilon}},\overrightarrow{\boldsymbol{\gamma}}_{\varepsilon})$, so that $(u_{\overrightarrow{\boldsymbol{\gamma}}_{\varepsilon}},\overrightarrow{\boldsymbol{\gamma}}_{\varepsilon})$ is minimizing $E^\mu_\varepsilon$ in $(1+H^1_0(\Omega))\times  \mathscr{P}(a_0,\mu)$. In particular, $\overrightarrow{\boldsymbol{\gamma}}_{\varepsilon}\in\mathscr{P}_{\Lambda_\varepsilon}(a_0,\mu)$ by Theorem  \ref{existminpairs}, 
and the proof is complete.  
\end{proof}


\section{The case of a general finite measure}\label{genmeassec}


 \subsection{Existence and regularity for a general finite measure}
 
 We consider in this subsection an arbitrary  (non negative) finite measure $\mu$ supported in $\overline\Omega_0$, and we fix a base point $a_0\in\overline\Omega_0$. We are interested in existence and regularity of solutions of the minimization problem
 \begin{equation}\label{minprobgen}
 \min_{u\in 1+H^1_0(\Omega)\cap L^\infty(\Omega)}F^\mu_\varepsilon(u) \,.
 \end{equation}
 To pursue these issues, we rely on the results of the previous section. For this, we will need the following elementary lemma. 
 
 \begin{lemma}\label{approxmeas}
Let $\mu$ be a finite non negative measure supported on $\overline\Omega_0$. Then there exists a sequence of measures $\{\mu_k\}_{k\in\mathbb{N}}$ with finite support in $\overline\Omega_0$    such that $\mu_k\mathop{\rightharpoonup}\limits^*\mu$ and ${\rm spt}\,\mu_k\to{\rm spt}\,\mu$ in the Hausdorff  sense. 
 \end{lemma}
 
 \begin{proof} 
For $k \in \N$, we denote by $\mathscr{C}_k$ be the standard family of dyadic semi-cubes in $\R^2$ of size $2^{-k}$, i.e.,
 $$\mathscr{C}_k:=\Big\{Q=2^{-k}z+2^{-k}\big([0,1)\times[0,1)\big): z\in \mathbb{Z}^2\Big\}\,.$$
 Then we define $\mathscr{C}_k^\prime:=\big\{ Q \in \mathscr{C}_k : Q\cap \overline\Omega_0 \not = \emptyset \big \}$, and  
 for each $Q\in \mathscr{C}_k^\prime$, we choose a point $a_Q\in Q\cap \overline\Omega_0$. We set 
 $$\mu_k:=\sum_{Q \in \mathscr{C}_k^\prime} \mu(Q\cap \overline\Omega_0) \delta_{a_Q}\,.$$
By construction, $\mu_k$ has finite support, $\|\mu_k\|=\|\mu\|$, and ${\rm spt}\,\mu_k\subset \overline\Omega_0\cap \mathcal{T}_{2^{-k+2}}({\rm spt}\,\mu)$ where $\mathcal{T}_{2^{-k+2}}({\rm spt}\,\mu)$ denotes the tubular neighborhood of radius $2^{-k+1}$ of ${\rm spt}\,\mu$. Similarly, ${\rm spt}\,\mu\subset  \mathcal{T}_{2^{-k+2}}({\rm spt}\,\mu_k)$, and we infer that  ${\rm spt}\,\mu_k\to{\rm spt}\,\mu$ in the Hausdorff  sense.

 We now claim that $\mu_k\mathop{\rightharpoonup}\limits^*\mu$ as measures on $\overline\Omega_0$. To prove this claim, let us fix an arbitrary function $\varphi \in C^0(\overline\Omega_0)$. Then we can find a (non decreasing) modulus of continuity $\omega:[0,\infty)\to[0,\infty)$ satisfying $\omega(t)\to 0$ as $t\downarrow 0$ such that 
$$\sup_{|x-y|\leq t}|\varphi(x)-\varphi(y)|\leq \omega(t)\,. $$
Now we estimate  
 $$\left| \int\varphi d \mu_k - \int \varphi d \mu \right| \leq \sum_{Q \in \mathscr{C}_k^\prime}\int_{Q\cap \overline\Omega_0} \big| \varphi(a_Q)-  \varphi(x) \big|\,d\mu \leq \|\mu\| \,\omega(2^{-k+1})\mathop{\longrightarrow}\limits_{k\to\infty} 0\,,$$
 which completes the proof.
\end{proof}

 \begin{theorem}\label{thmexisgen}
 The minimization problem \eqref{minprobgen} admits at least one solution. 
 \end{theorem}

 \begin{proof}
 We consider  the sequence of discrete measures $\{\mu_k\}_{k\in\mathbb{N}}$ provided by Lemma~\ref{approxmeas}.
 For each $k\in\mathbb{N}$, we consider a solution $u_k$ of the minimization problem
 $$\min_{u\in 1+H^1_0(\Omega)} F^{\mu_k}_\varepsilon(u) \,,$$
for some base point $a^k_0\in \overline\Omega_0$ satisfying $a^k_0\to a_0$.   Since $\|\mu_k\|$ is bounded, by Proposition~\ref{Holder}, the sequence $\{u_k\}_{k\in\mathbb{N}}$ is bounded in $C^{0,\alpha}(\Omega)$ for every $\alpha\in(0,1)$, and $0\leq u_k\leq 1$. Moreover, choosing a ($k$-independent) $C^1$-function to test the minimality of $u_k$, we infer that $F^{\mu_k}_\varepsilon(u_k)\leq C$ for some constant $C$ independent of $k$. As a consequence,  $\{u_k\}_{k\in\mathbb{N}}$ is bounded in $H^1(\Omega)$. Therefore, we can find a (not relabeled) subsequence such that $u_k\to u_*$ in $C^{0,\alpha}(\Omega)$ for every $\alpha\in(0,1)$ and  $u_k\rightharpoonup u_*$ weakly in $H^1(\Omega)$. Then, $u_*\in 1+H^1_0(\Omega)$ and 
\begin{equation}\label{limi1735}
\liminf_{k\to\infty} \varepsilon\int_{\Omega}|\nabla u_k|^2\,dx+\frac{1}{4\varepsilon}\int_{\Omega}(1-u_k)^2\,dx\geq \varepsilon\int_{\Omega}|\nabla u_*|^2\,dx+\frac{1}{4\varepsilon}\int_{\Omega}(1-u_*)^2\,dx\,.
\end{equation}
We now claim that the sequence of continuous  functions ${\bf d}_k:x\mapsto {\bf D}(\delta_\varepsilon+u_k^2;a^k_0,x)$ converges uniformly on $\Omega$ to ${\bf d}_*:x\mapsto {\bf D}(\delta_\varepsilon+u_*^2;a_0,x)$. Since $\|u_k\|_{L^\infty(\Omega)}\leq 1$, each function ${\bf d}_k$ is $(1+\delta_\eps)$-Lipschitz continuous. Hence the sequence $\{{\bf d}_k\}_{k\in\mathbb{N}}$ is uniformly equicontinuous, and it is enough to prove that ${\bf d}_k$ converges pointwise to ${\bf d}_*$. Let us then fix an arbitrary point $x\in\Omega$. For $\gamma\in\mathscr{P}(a_0,x)$, we have 
\begin{multline*}
{\bf d}_k(x)\leq  {\bf D}(\delta_\varepsilon+u_k^2;a_0,x)+(1+\delta_\eps)|a_0^k-a_0| \\
\leq  \int_{\Gamma(\gamma)}(\delta_\eps+u_k^2)\,d\mathcal{H}^1 +(1+\delta_\eps)|a_0^k-a_0|\,,
\end{multline*}
and we obtain by dominated convergence, 
$$\limsup_{k\to\infty}{\bf d}_k(x)\leq  \int_{\Gamma(\gamma)}(\delta_\eps+u_*^2)\,d\mathcal{H}^1 \,. $$
Taking the infimum over $\gamma$ shows that $\limsup_k {\bf d}_k(x)\leq {\bf d}_*(x)$. On the other hand, if $\sigma\in(0,1)$, we can find $\gamma_k\in\mathscr{P}(a^k_0,x)$ such that 
$$\int_{\Gamma(\gamma_k)}(\delta_\eps+u_k^2)\,d\mathcal{H}^1 \leq  {\bf d}_k(x)+\sigma\,.$$
In particular, $\mathcal{H}^1(\Gamma(\gamma_k))\leq \delta_\eps^{-1}( {\bf d}_k(x)+\sigma)\leq C$. Since $u_k\to u_*$ uniformly, we have $u_k^2\geq u_*^2-\sigma$ whenever $k$ is large enough. For such $k$'s, we estimate 
\begin{multline*}
 {\bf d}_k(x) \geq \int_{\Gamma(\gamma_k)}(\delta_\eps+u_*^2)\,d\mathcal{H}^1 -\big(1+\mathcal{H}^1(\Gamma(\gamma_k))\big)\sigma\\
 \geq   {\bf D}(\delta_\varepsilon+u_*^2;a^k_0,x)-C\sigma\geq {\bf d}_*(x)-(1+\delta_\eps)|a_0^k-a_0|-C\sigma\,.
 \end{multline*}
 Letting  $k\uparrow\infty$ and then $\sigma\downarrow 0$, we deduce that $\liminf_k {\bf d}_k(x) \geq {\bf d}_*(x)$, whence ${\bf d}_k(x) \to  {\bf d}_*(x)$.

Now, as a consequence of this uniform convergence, we have
\begin{equation}\label{kloklo}
\int_{\overline\Omega_0} {\bf D}(\delta_\varepsilon+u_k^2;a_0^k,x)\,d\mu_k\longrightarrow  \int_{\overline\Omega_0} {\bf D}(\delta_\varepsilon+u_*^2;a_0,x)\,d\mu\,.
\end{equation}
Gathering \eqref{limi1735} and \eqref{kloklo} leads to 
$$\liminf_{k\to\infty} F^{\mu_k}_\varepsilon(u_k)\geq F^\mu_\varepsilon(u_*)\,. $$
To conclude, we consider an arbitrary $\varphi\in 1+H^1_0(\Omega)\cap L^\infty(\Omega)$. Since 
$$\big|{\bf D}(\delta_\varepsilon+\varphi^2;a_0,x)-{\bf D}(\delta_\varepsilon+\varphi^2;a^k_0,x)\big|\leq (\delta_\eps+\|\varphi\|^2_{L^\infty(\Omega)})|a_0^k-a_0|\to 0\,,$$
we have  $\int {\bf D}(\delta_\varepsilon+\varphi^2;a^k_0,x)\,d\mu_k\to \int {\bf D}(\delta_\varepsilon+\varphi^2;a_0,x)\,d\mu$, and thus $F_\eps^{\mu_k}(\varphi)\to F_\eps^{\mu}(\varphi)$.  By minimality of $u_k$, we conclude that 
$$ F^\mu_\varepsilon(u_*)\leq \liminf_{k\to\infty} F^{\mu_k}_\varepsilon(u_k)\leq \limsup_{k\to\infty} F^{\mu_k}_\varepsilon(u_k)\leq \lim_{k\to\infty}F_\eps^{\mu_k}(\varphi)=  F_\eps^{\mu}(\varphi)\,. $$
Consequently, $u_*$ is minimizing $F^\mu_\varepsilon$, and (choosing $\varphi=u_*$) $F_\eps^{\mu_k}(u_k)\to F_\eps^{\mu}(u_*)$. For later use, we also observe that the $\liminf$ in \eqref{limi1735} now becomes a limit (in view of \eqref{kloklo}), and the inequality turns into an equality, i.e., 
$$\lim_{k\to\infty} \varepsilon\int_{\Omega}|\nabla u_k|^2\,dx+\frac{1}{4\varepsilon}\int_{\Omega}(1-u_k)^2\,dx= \varepsilon\int_{\Omega}|\nabla u_*|^2\,dx+\frac{1}{4\varepsilon}\int_{\Omega}(1-u_*)^2\,dx\,. $$ 
From this identity, it classicaly follows that $u_k\to u_* $ strongly in $H^1(\Omega)$. 
 \end{proof}

 Note that the previous proof not only produces a minimizer of $F_\eps^\mu$, but it produces a $W^{1,p}$-minimizer. Our next theorem shows that, in fact,  any minimizer shares the same regularity. 
 
 \begin{theorem}\label{mainmain}
 Any solution of the minimization problem \eqref{minprobgen} belongs to $W^{1,p}(\Omega)$ for every $p<\infty$ (and in particular to $C^{0,\alpha}(\Omega)$ for every $\alpha\in(0,1)$). 
 \end{theorem}
 
 \begin{proof}
 Consider $u_*$ a solution of \eqref{minprobgen}. First we claim that $0\leq u_*\leq 1$ a.e. in $\Omega$. Indeed, if this is not the case, then we consider the competitor $\bar u:=\max(\min(u_*,1),0)$. Arguing as in the proof of Lemma \ref{bound1}, we have ${\bf D}(\delta_\eps+(\bar u)^2;a_0,x)\leq {\bf D}(\delta_\eps+u_*^2;a_0,x)$ for every $x\in\Omega$. 
 Then, as in the proof of Corollary \ref{thmminF}, it leads to $F_\eps^{\mu}(\bar u)<F_\eps^{\mu}(u_*)$, in contradiction with the minimality of $u_*$. 

Now the strategy consists in introducing the modified functionals $\widehat F^{\mu}_\eps:H^1(\Omega)\cap L^\infty(\Omega)\to[0,\infty)$ defined by 
$$\widehat F^{\mu}_\eps(u):=F^{\mu}_\eps(u)+\frac{1}{4}\int_\Omega|u-u_*|^2\,dx\,. $$ 
Since $u_*$ is minimizing $F^{\mu}_\eps$, it is also the {\it unique} minimizer of $\widehat F^{\mu}_\eps$ over $1+H^1_0(\Omega)\cap L^\infty(\Omega)$. 
Then we consider the sequence of discrete measures $\{\mu_k\}_{k\in\mathbb{N}}$ provided by Lemma~\ref{approxmeas}, and the corresponding functionals $\widehat F^{\mu_k}_\eps:H^1(\Omega)\cap L^\infty(\Omega)\to[0,\infty)$ given by 
$$\widehat F^{\mu_k}_\eps(u):=F^{\mu_k}_\eps(u)+\frac{1}{4}\int_\Omega|u-u_*|^2\,dx\,, $$
with base point $a_0^k\in{\rm spt}\,\mu_k$. We aim to address the minimization problems
\begin{equation}\label{pbpertk}
\min_{u\in 1+H^1_0(\Omega)\cap L^\infty(\Omega)} \widehat  F^{\mu_k}_\eps(u)\,.
\end{equation}
We shall prove existence and regularity of minimizers for \eqref{pbpertk} following the main lines of Section \ref{newsection}. More precisely, we will prove that the $W^{1,p}$-norm of a constructed  minimizer $u_k$ of $\widehat F^{\mu_k}_\eps$ remains bounded for every $p<\infty$ independently of $k$ (and thus also the $C^{0,\alpha}$-norm for every $\alpha\in(0,1)$). Assuming that this is indeed the case, we can run the proof of Theorem \ref{thmexisgen} noticing the additional term $\|u-u_*\|^2_{L^2(\Omega)}$  is   continuous with respect to weak $H^1$-convergence. In other words, we can extract from the resulting sequence $\{u_k\}_{k\in\mathbb{N}}$, a subsequence converging strongly in $H^1(\Omega)$ (and in $C^{0,\alpha}$) to a limiting function $u_0\in 1+H^1_0(\Omega)\cap L^\infty(\Omega)$ minimizing $ \widehat F^{\mu}_\eps$.  Since $u_*$ is the unique minimizer  of  $ \widehat F^{\mu}_\eps$ over $1+H^1_0(\Omega)\cap L^\infty(\Omega)$, we have $u_0=u_*$ and $u_k\to u_*$. Finally, since $\{u_k\}_{k\in\mathbb{N}}$ remains bounded in $W^{1,p}(\Omega)$, it shows that $u_*\in W^{1,p}(\Omega)$ for every $p<\infty$. 
\vskip3pt

Now comes the analysis of problem \eqref{pbpertk}: 
\vskip3pt

\noindent{\it Step 1: Minimization with prescribed curves.} We write 
$$\mu_k=\sum^{N_k}_{i=0}\beta^k_{i}\delta_{a_i^k}\,,$$
with $\beta_i^k>0$. 
For $\overrightarrow{\boldsymbol{\gamma}}\in \mathscr{P}(a_0^k,\mu_k)$, we consider the  functional $\widehat E^{\mu_k}_\varepsilon(\cdot,\overrightarrow{\boldsymbol{\gamma}}):H^1(\Omega)\to [0,+\infty]$ defined by 
$$\widehat E^{\mu_k}_\varepsilon(u,\overrightarrow{\boldsymbol{\gamma}}):=E^{\mu_k}_\varepsilon(u,\overrightarrow{\boldsymbol{\gamma}})+\frac{1}{4}\int_\Omega|u-u_*|^2\,dx\,, $$
where $E^{\mu_k}_\varepsilon(u,\overrightarrow{\boldsymbol{\gamma}})$ is given by \eqref{defEmu}. Then,  
\begin{equation}\label{samsoir1740}
\widehat F^{\mu_k}_\eps(u)=\inf_{\overrightarrow{\boldsymbol{\gamma}}\in \mathscr{P}(a_0^k,\mu_k)} \widehat E^{\mu_k}_\varepsilon(u,\overrightarrow{\boldsymbol{\gamma}})\qquad\forall u\in H^1(\Omega)\cap L^\infty(\Omega)\,.
\end{equation}
Let us now fix  $\overrightarrow{\boldsymbol{\gamma}}\in \mathscr{P}_\Lambda(a_0,\mu_k)$ for some $\Lambda\ge 2$. By Lemma \ref{lemfond}, the minimization problem 
$$\min_{u\in 1+H^1_0(\Omega)} \widehat E^{\mu_k}_\varepsilon(u,\overrightarrow{\boldsymbol{\gamma}}) $$ 
admits a unique solution  $\widehat  u_{\overrightarrow{\boldsymbol{\gamma}}}$ solving 
$$\begin{cases}
\displaystyle -\eps^2\Delta \widehat u_{\overrightarrow{\boldsymbol{\gamma}}}= \frac{1}{4} (1- \widetilde u_{\overrightarrow{\boldsymbol{\gamma}}}) +\frac{\eps}{4}(u_*- \widehat u_{\overrightarrow{\boldsymbol{\gamma}}})-\frac{\varepsilon}{\lambda_\varepsilon}B_\mu[\overrightarrow{\boldsymbol{\gamma}}](\widehat u_{\overrightarrow{\boldsymbol{\gamma}}},\cdot) &\text{in $H^{-1}(\Omega)$}\,,\\[8pt]
 \widehat u_{\overrightarrow{\boldsymbol{\gamma}}}= 1 & \text{on $\partial\Omega$}\,.
\end{cases}$$
In addition, since $0\leq u_*\leq 1$, the truncation argument in the proof of Lemma \ref{bound1} shows that $0\leq \widehat u_{\overrightarrow{\boldsymbol{\gamma}}}\leq 1$ a.e. in $\Omega$. As a consequence, $|u_*- \widehat u_{\overrightarrow{\boldsymbol{\gamma}}}|\leq 1$ a.e. in $\Omega$. By elliptic regularity, we then infer that 
 $ \widehat u_{\overrightarrow{\boldsymbol{\gamma}}} \in C^{1,\alpha}_{\rm loc}\big(\overline\Omega \setminus \Gamma(\overrightarrow{\boldsymbol{\gamma}})\big)$ for every $\alpha\in(0,1)$.
 
Considering the function $\widehat v:=1- \widehat u_{\overrightarrow{\boldsymbol{\gamma}}} $, we notice that 
 $$ \begin{cases} 
 -4\eps^2\Delta \widehat v+\widehat v\leq \eps & \text{in $\Omega\setminus \Gamma(\overrightarrow{\boldsymbol{\gamma}})$}\,,\\
 0\leq \widehat v\leq 1 & \text{in $\Omega$}\,. 
\end{cases}$$
Then a straightforward modification of Lemma  \ref{comp} shows that 
$$0\leq 1- \widehat u_{\overrightarrow{\boldsymbol{\gamma}}}(x_0)\leq \eps +\exp\left(-\frac{3\,{\rm dist}(x_0,\Gamma(\overrightarrow{\boldsymbol{\gamma}}))}{32\varepsilon}\right)$$ 
at every $x_0\in \Omega\setminus \Gamma(\overrightarrow{\boldsymbol{\gamma}})$ satisfying ${\rm dist}(x_0,\Gamma(\overrightarrow{\boldsymbol{\gamma}}))\geq 12\varepsilon$. As in Lemma \ref{estigradhessloin}, this leads to the gradient estimate 
\begin{equation}\label{pwtgradpert}
\big|\nabla \widehat u_{\overrightarrow{\boldsymbol{\gamma}}}(x_0)\big| \leq C_{\boldsymbol{\eta}_0}\left( 1+\frac{1}{\varepsilon}  \exp\left(-\frac{{\rm dist}(x_0,\Gamma(\overrightarrow{\boldsymbol{\gamma}}))}{32\varepsilon}\right)\right)
\end{equation}
at every $x_0\in \overline\Omega\setminus \Gamma(\overrightarrow{\boldsymbol{\gamma}})$ satisfying ${\rm dist}(x_0,\Gamma(\overrightarrow{\boldsymbol{\gamma}}))\geq 13\varepsilon$ (with $\boldsymbol{\eta}_0$ given by Lemma \ref{estigradhessloin}).

Since $\|u_*-\widehat u_{\overrightarrow{\boldsymbol{\gamma}}}\|_{L^\infty(\Omega)}\leq 1$, we can reproduce the proof of Proposition \ref{W1p} with minor modifications  to prove that 
$\widehat u_{\overrightarrow{\boldsymbol{\gamma}}}\in W^{1,p}(\Omega)$ for every $2<p<\infty$ together with the estimate
$$\|\nabla \widehat u_{\overrightarrow{\boldsymbol{\gamma}}}\|_{L^p(V_{32\eps|\log\eps|})} \\
 \leq C_{p,\boldsymbol{\eta}_0}\left(\frac{|\log\eps|}{\varepsilon}+\frac{\Lambda\|\mu_k\|}{\lambda_\varepsilon\varepsilon}\right)\,,$$
where $V_{32\eps|\log\eps|}:=\{x\in\Omega:{\rm dist}(x,\Gamma(\overrightarrow{\boldsymbol{\gamma}}))<32\eps|\log\eps|\}$. On the other hand, \eqref{pwtgradpert} yields the estimate  $|\nabla \widehat u_{\overrightarrow{\boldsymbol{\gamma}}}|\leq C_{\boldsymbol{\eta}_0}$ on $\Omega\setminus V_{32\eps|\log\eps|}$. Therefore, 
$$\|\nabla \widehat u_{\overrightarrow{\boldsymbol{\gamma}}}\|_{L^p(\Omega)} \\
 \leq C_{p,\boldsymbol{\eta}_0}\left(\frac{|\log\eps|}{\varepsilon}+\frac{\Lambda\|\mu_k\|}{\lambda_\varepsilon\varepsilon}\right)\quad\text{for  $2<p<\infty$}\,. $$ 
 Since $\|\mu_k\|$ is bounded, we have thus proved that $\|\widehat u_{\overrightarrow{\boldsymbol{\gamma}}}\|_{W^{1,p}(\Omega)}$ is bounded independently of $k$ for each $p<\infty$. 
 \vskip3pt

\noindent{\it Step 2: Existence of minimizing pairs.} Define $\Lambda_\eps$ as in \eqref{Lambdaeps}. Then we  notice that Lemma \ref{replace}, Lemma \ref{geodlem}, and Lemma \ref{goodpairs} hold with $\widehat E^{\mu_k}_\varepsilon$ in place of  $E^{\mu_k}_\varepsilon$. Hence we can follow the proof of Theorem \ref{existminpairs} to find  
$\overrightarrow{\boldsymbol{\gamma}}_k\in \mathscr{P}_{\Lambda_\eps}(a_0^k,\mu_k)$ such that the pair $(\widehat u_{\overrightarrow{\boldsymbol{\gamma}}_k},\overrightarrow{\boldsymbol{\gamma}}_k)$ is minimizing $\widehat E^{\mu_k}_\varepsilon$ over $(1+H^1_0(\Omega))\times \mathscr{P}(a_0^k,\mu_k)$. 

 \vskip3pt

\noindent{\it Step 3: Conclusion.} Set $u_k:=\widehat  u_{\overrightarrow{\boldsymbol{\gamma}}_k}$. Since $u_k\in L^\infty(\Omega)$, we infer from \eqref{samsoir1740} that $\widehat F^{\mu_k}_\eps(u_k)= \widehat E^{\mu_k}_\varepsilon(u_k,\overrightarrow{\boldsymbol{\gamma}}_k)$, and thus $u_k$ is minimizing $\widehat F^{\mu_k}_\eps$ over 
$1+H^1_0(\Omega)\cap L^\infty(\Omega)$.  Finally, it follows from Step 1 that $\|u_k\|_{W^{1,p}(\Omega)}$ is bounded independently of $k$ for every $p<\infty$.   
\end{proof}
 
 \begin{remark}
 The proof of Theorem \ref{mainmain} (together with the results in Subsection \ref{subsecprescrcurv}) shows that any minimizer $u_\eps$ of $F_\eps^\mu$ over $1+H^1_0(\Omega)\cap L^\infty(\Omega)$ satisfies the following estimates
 $$ \|\nabla u_\eps\|_{L^p(\Omega)}\leq C_{p,\boldsymbol{\eta}_0}\left(\frac{|\log\eps|}{\varepsilon}+\frac{\|\mu\|}{\delta_\eps\lambda_\varepsilon\varepsilon}\right)\qquad\forall p\in(2,\infty)\,,$$
 and 
 $$\|u_\eps\|_{C^{0,\alpha}(\Omega)}\leq C_{\alpha,\boldsymbol{\eta}_0}\frac{1+\|\mu\|\delta_\eps^{-1}\lambda^{-1}_\varepsilon}{\varepsilon^\alpha} \qquad\forall \alpha\in(0,1)\,, $$
 for some constants $C_{p,\boldsymbol{\eta}_0}$ and $C_{\alpha,\boldsymbol{\eta}_0}$ depending only on $p$, $\alpha$, and $\boldsymbol{\eta}_0$ (given in  Lemma \ref{estigradhessloin}). Even if those estimates are not optimal with respect to $\eps$ (but nearly), they only depends on the total mass of $\mu$, and not on the internal structure of $\mu$. 
 \end{remark}
 
  In view of the uniform estimates above, one can reproduce (verbatim) the proof of Theorem~\ref{thmexisgen} to show the following stability result.
 
 \begin{proposition}\label{limitweak}
 Let $\{\mu_k\}_{k\in\mathbb{N}}$ be a sequence of finite measures supported on $\overline\Omega_0$, and $\{a_0^k\}_{k\in\mathbb{N}}\subset \overline\Omega_0$.  Assume that $\mu_k\mathop{\rightharpoonup}\limits^*\mu$ as measures and $a_0^k\to a_0$. If $u_k$ is a minimizer of $F_\eps^{\mu_k}$ with base point $a_0^k$ over $1+H^1_0(\Omega)\cap L^\infty(\Omega)$, then the sequence $\{u_k\}_{k\in\mathbb{N}}$ admits a (not relabeled) subsequence converging strongly in $H^1(\Omega)$ and  in $C^{0,\alpha}(\Omega)$ for every $\alpha\in(0,1)$ to a minimizer $u_*$ of $F_\eps^{\mu}$ with base point $a_0$ over $1+H^1_0(\Omega)\cap L^\infty(\Omega)$. In addition, $F_\eps^{\mu_k}(u_k)\to F_\eps^{\mu}(u_*)$. 
 \end{proposition}

 \subsection{Application to the average distance and optimal compliance problems}\label{AverdistCompsec}
 
In this subsection, we briefly review and complement two applications suggested in \cite{BLS}:  the average distance problem and the optimal compliance problem. 
\vskip3pt

\noindent{\it (1) The average distance problem.} Given a nonnegative density $f\in L^1(\Omega_0)$,  it consists in finding a connected compact set $K_\sharp\subset \overline\Omega_0$ minimizing the functional 
$${\bf AVD}(K):=\int_{\Omega_0}{\rm dist}(x,K)f(x)\,dx+\mathcal{H}^1(K) $$
among all connected and compact subsets $K$ of $\overline\Omega_0$. 
\vskip3pt

\noindent{\it (2) The optimal compliance problem.} Given a nonnegative $f\in L^2(\Omega_0)$,  it consists in finding a connected compact set $K_\sharp\subset \overline\Omega_0$ minimizing the functional 
$${\bf OPC}(K):=\frac{1}{2}\int_{\Omega_0}  f u_K\,dx +\mathcal{H}^1(K) $$
among all connected and compact subsets $K$ of $\overline\Omega_0$ of positive $\mathcal{H}^1$-measure, where $u_K\in H^1(\Omega_0)$ denotes the unique solution of the minimization problem
$$\min\Big\{\frac{1}{2}\int_{\Omega_0}|\nabla u|^2\,dx-\int_{\Omega_0}fu\,dx: u\in H^1(\Omega_0)\,,\;u=0\text{ on $K$}\Big\}\,. $$
\vskip3pt
 
\noindent{\it Reformulating problems (1) and (2).} The starting point in \cite{BLS} is a suitable reformulation of the average distance and optimal compliance problems by a duality argument. To describe in detail these reformulations, we need first to introduced the functional spaces involved. We fix a {\sl base point} $a_0\in\overline\Omega_0$. Setting $\mathscr{M}(\overline\Omega_0)$, respectively $\mathscr{M}(\overline\Omega_0;\mathbb{R}^2)$, the space of (finite) $\R$-valued, respectively  $\R^2$-valued, measures on $\R^2$ supported on $\overline\Omega_0$, we consider the following families of (generalized) vector fields
$$\mathscr{V}_{\rm avd}(\Omega_0):=\Big\{v\in \mathscr{M}(\overline\Omega_0;\mathbb{R}^2): {\rm div}\,v\in\mathscr{M}(\overline\Omega_0) \ {\rm and}\  {\rm div}\,v(\overline\Omega_0)=0\Big\}\,, $$
and
$$\mathscr{V}_{\rm opc}(\Omega_0):=\Big\{v\in L^2(\Omega_0;\mathbb{R}^2):{\rm div}(\chi_{\Omega_0}v) \in\mathscr{M}(\overline\Omega_0)  \ {\rm and}\  {\rm div}(\chi_{\Omega_0}v)(\overline\Omega_0)=0 \Big\}\,. $$
For such a vector field $v$, we associate the (finite) nonnegative measure
$$\mu(v):=\begin{cases} 
|{\rm div}\,v+\chi_{\Omega_0}f| & \text{if $v\in \mathscr{V}_{\rm avd}(\Omega_0)$}\,,\\
|{\rm div}(\chi_{\Omega_0}v)+\chi_{\Omega_0}f| & \text{if $v\in\mathscr{V}_{\rm opc}(\Omega_0)$}\,.
\end{cases}$$
We define the pointed functionals $\mathscr{F}_{\rm avd}:\overline\Omega_0\times\mathscr{M}(\overline\Omega_0;\R^2)\to[0,\infty]$ and $\mathscr{F}_{\rm opc}:\overline\Omega_0\times L^2(\Omega_0;\R^2)\to[0,\infty]$ by 
$$\mathscr{F}_{\rm avd}(a_0,v):=\begin{cases}
\|v\|+\|{\rm div}\,v\|+ \mathscr{S}\big(\{a_0\}\cup{\rm spt}\,\mu(v)\big) & \text{if $v\in \mathscr{V}_{\rm avd}(\Omega_0)$}\,,\\
+\infty & \text{otherwise}\,,
\end{cases}$$
and 
$$\mathscr{F}_{\rm opc}(a_0,v):=\begin{cases}
\displaystyle \frac{1}{2}\int_{\Omega_0}|v|^2\,dx+\|{\rm div}\,v\|+ \mathscr{S}\big(\{a_0\}\cup{\rm spt}\,\mu(v)\big) & \text{if $v\in \mathscr{V}_{\rm opc}(\Omega_0)$}\,,\\[5pt] 
+\infty & \text{otherwise}\,,
\end{cases}$$ 
where  $\|v\|$ and $\|{\rm div}\,v\|$ denote the total variations of $v$ and ${\rm div}\,v$, and 
$$\mathscr{S}\big(\{a_0\}\cup{\rm spt}\,\mu(v)\big):=\inf\Big\{\mathcal{H}^1(K): K\subset\overline\Omega_0 \text{ compact connected, } K\supset \{a_0\}\cup{\rm spt}\,\mu(v)\Big\} $$
(the infimum being infinite if the class of competitors is empty). 

Following \cite[proof of Proposition 5.6]{BLS}, the variational problems 
$$\min_{a_0\in\overline\Omega_0}\Big(\min_{\mathscr{V}_{\rm avd}(\Omega_0)}\mathscr{F}_{\rm avd}(v,a_0)\Big)  \quad \text{and} \quad \min_{a_0\in\overline\Omega_0}\Big(\min_{\mathscr{V}_{\rm opc}(\Omega_0)}\mathscr{F}_{\rm opc}(v,a_0) \Big)$$
admit at least one solution $(a^\sharp_0,v^\sharp_{\rm avd})$ and $(a^\sharp_0,v^\sharp_{\rm opc})$, respectively. According to \cite[Section 5.1]{BLS}, their resolution is equivalent to problems (1) and (2), respectively\footnote{In the original formulation of \cite{BLS}, one requires $a_0\in{\rm spt}\,\mu(v)$ in the definition $\mathscr{F}_{\rm avd}(a_0,v)$ or $\mathscr{F}_{\rm opc}(a_0,v)$. A quick inspection of \cite[Section 5.1]{BLS} reveals that this condition can be dropped when considering $\mathscr{S}\big(\{a_0\}\cup{\rm spt}\,\mu(v)\big)$ instead of $\mathscr{S}\big({\rm spt}\,\mu(v)\big)$.}. As our purpose is not focused on this equivalent formulation, we only indicate the following implication: if $K_{\rm avd}^\sharp$ and  $K_{\rm opc}^\sharp$ are compact connected subsets of $\overline\Omega_0$ satisfying   
\begin{equation}\label{avion1}
\mathcal{H}^1(K_{\rm avd}^\sharp)=\mathscr{S}\big(\{a_0^\sharp\}\cup{\rm spt}\,\mu(v_{\rm avd}^\sharp)\big)\quad\text{and}\quad \mathcal{H}^1(K_{\rm opc}^\sharp)=\mathscr{S}\big(\{a_0^\sharp\}\cup{\rm spt}\,\mu(v_{\rm opc}^\sharp)\big)\,,
\end{equation}
then, 
\begin{equation}\label{avion2}
{\bf AVD}(K_{\rm avd}^\sharp)=\min\,{\bf AVD}\quad \text{and} \quad {\bf OPC}(K_{\rm opc}^\sharp)=\min\,{\bf OPC}\,. 
\end{equation}
In other words, $K_{\rm avd}^\sharp$ and  $K_{\rm opc}^\sharp$ solve problem (1) and problem (2) respectively. 
\vskip3pt

\noindent{\it The phase field approximation.} The phase field approximation introduced in \cite{BLS} to solve problem (1) or (2) consists in replacing the term $\mathscr{S}\big(\{a_0\}\cup{\rm spt}\,\mu(\cdot)\big)$ in $\mathscr{F}_{\rm avd}(\cdot,a_0)$ or  $\mathscr{F}_{\rm opc}(\cdot,a_0)$ by the functional $\widetilde F_\eps^{\mu(\cdot)}$ defined in \eqref{defftild}. As explained in the introduction (see also \cite[Section 5.4]{BLS}), the possible lack of lower semicontinuity of $\widetilde F_\eps^{\mu(\cdot)}$ prevents one to obtain existence of minimizers for the resulting phase field functionals. 

Here we follow the approach of \cite{BLS} using the functional $F_\eps^{\mu(\cdot)}$ instead of $\widetilde F_\eps^{\mu(\cdot)}$. More precisely, we consider the functionals $\mathscr{F}^\eps_{\rm avd}: \overline\Omega_0\times\mathscr{M}(\overline\Omega_0;\R^2)\times \big(1+H_0^1(\Omega)\cap L^\infty(\Omega)\big)\to [0,\infty]$ and 
$\mathscr{F}^\eps_{\rm opc}: \overline\Omega_0\times L^2(\Omega_0;\R^2)\times \big(1+H_0^1(\Omega)\cap L^\infty(\Omega)\big)\to [0,\infty]$ given by 
\begin{equation}\label{defpfavd}
 \mathscr{F}^\eps_{\rm avd}(a_0,v,u):=\begin{cases}
\|v\|+\|{\rm div}\,v\|+F^{\mu(v)}_\eps(u) & \text{if $v\in \mathscr{V}_{\rm avd}(\Omega_0)$}\,,\\
+\infty & \text{otherwise}\,,
\end{cases}
\end{equation}
and 
\begin{equation}\label{defpfopc}
 \mathscr{F}^\eps_{\rm opc}(a_0,v,u):=\begin{cases}
\displaystyle \frac{1}{2}\int_{\Omega_0}|v|^2\,dx+\|{\rm div}\,v\|+F^{\mu(v)}_\eps(u) & \text{if $v\in \mathscr{V}_{\rm opc}(\Omega_0)$}\,,\\
+\infty & \text{otherwise}\,,
\end{cases}
\end{equation}
where $a_0$ is the base point in $F^{\mu(v)}_\eps$. As a consequence of Theorem \ref{thmexisgen} and Proposition \ref{limitweak}, we have the following existence result of minimizers. Their convergence  as $\eps\to 0$ towards minimizers of $\mathscr{F}_{\rm avd}$ or $\mathscr{F}_{\rm opc}$ (essentially proved in \cite{BLS}) shall be discussed for completeness in Subsection~\ref{asymptavdopc}. 

\begin{theorem}
The functionals $ \mathscr{F}^\eps_{\rm avd}$ and $\mathscr{F}^\eps_{\rm opc}$ admit at least one minimizer. 
\end{theorem}

\begin{proof}
First notice that, for $a\in\overline\Omega_0$, the competitor $(a,0,1)$ has a finite energy, so that the infimum of $ \mathscr{F}^\eps_{\rm avd}$ and $\mathscr{F}^\eps_{\rm opc}$ are finite.  Let us now consider an arbitrary minimizing sequence $\{(a_0^k,v_k,\widetilde u_k)\}_{k\in\mathbb{N}}$ for $ \mathscr{F}^\eps_{\rm avd}$ or $ \mathscr{F}^\eps_{\rm opc}$. By Theorem \ref{thmexisgen}, we can find for each $k\in\mathbb{N}$ a minimizer $u_k$ of $F_\eps^{\mu(v_k)}$ with base point $a_0^k$ over $1+H^1_0(\Omega)\cap L^\infty(\Omega)$. Then, 
$F_\eps^{\mu(v_k)}(u_k)\leq F_\eps^{\mu(v_k)}(\widetilde u_k)$, so that  $\{(a_0^k,v_k,u_k)\}_{k\in\mathbb{N}}$ is also a minimizing sequence. 
\vskip3pt

\noindent{\it Case 1: minimizing $ \mathscr{F}^\eps_{\rm avd}$.} Since  $\sup_k  \mathscr{F}^\eps_{\rm avd}(a_0^k,v_k,u_k)<\infty$, we can find a (not relabeled) subsequence such that $v_k\mathop{\rightharpoonup}\limits^* v_\eps$ and ${\rm div}\,v_k\mathop{\rightharpoonup}\limits^* {\rm div}\,v_\eps$ as measures for some $v_\eps\in\mathscr{V}_{\rm avd}$ (note that the divergence free condition is closed under those weak* convergences), and $a_0^k\to a_0^\eps$ for some $a_0^\eps\in\overline\Omega_0$. Since $\mu(v_k)\mathop{\rightharpoonup}\limits^*\mu(v_\eps)$, we infer from Proposition \ref{limitweak} that (up to a further subsequence) $u_k\to u_\eps$ strongly in $H^1(\Omega)$ to some minimizer $u_\eps$ of  $F_\eps^{\mu(v_\eps)}$ with base point $a_0^\eps$ over $1+H^1_0(\Omega)\cap L^\infty(\Omega)$, and $F_\eps^{\mu(v_k)}(u_k)\to F_\eps^{\mu(v_\eps)}(u_\eps)$. Since the total variation is lower semicontinuous with respect to the weak* convergence of measures, we can now deduce that 
$$\mathscr{F}^\eps_{\rm avd}(a_0^\eps,v_\eps,u_\eps)\leq \lim_{k\to\infty} \mathscr{F}^\eps_{\rm avd}(a_0^k,v_k,u_k) = \inf  \mathscr{F}^\eps_{\rm avd}\,, $$
and $(a_0^\eps,v_\eps,u_\eps)$ is a minimizer of $\mathscr{F}^\eps_{\rm avd}$. 
\vskip3pt

\noindent{\it Case 2: minimizing $ \mathscr{F}^\eps_{\rm opc}$.} We argue as in Case 1, replacing the weak* convergence of the $v_k$'s by the weak convergence in $L^2(\Omega_0)$. 
\end{proof}

\begin{remark}
If $(a_0^\eps,v_\eps,u_\eps)$ is a minimizer of $\mathscr{F}^\eps_{\rm avd}$ or $\mathscr{F}^\eps_{\rm opc}$, then $u_\eps$ is a minimizer of  $F_\eps^{\mu(v_\eps)}$ with base point $a_0^\eps$ over $1+H^1_0(\Omega)\cap L^\infty(\Omega)$. Therefore, $u_\eps\in W^{1,p}(\Omega)$ for every $p<\infty$ (in particular, $u_\eps\in C^{0,\alpha}(\Omega)$ for every $\alpha\in(0,1)$). We did not investigate the regularity of the vector field $v_\eps$, and this question remains essentially open. 
\end{remark}

 
 \section{Asymptotic of minimizers}\label{sectasympt}  
 

\subsection{Towards the Steiner problem}

 The objective of this subsection is to prove Theorem \ref{thmmain2}. We start with elementary comments about the Steiner problem \eqref{classStein}.  Setting 
$$\mathscr{S}(\{a_0\}\cup{\rm spt}\,\mu):=\inf\Big\{\mathcal{H}^1(K): K\subset\R^2 \text{ compact connected, } K\supset\{a_0\}\cup{\rm spt}\,\mu \Big\} \,, $$
one has $\mathscr{S}(\{a_0\}\cup{\rm spt}\,\mu)<\infty$ if and only if $\mathcal{H}^1({\rm spt}\,\mu)<\infty$.  In addition, if we denote by $\boldsymbol{\pi}_0$ the orthogonal projection on the convex set $\overline\Omega_0$, then $\mathcal{H}^1(\boldsymbol{\pi}_0(K))\leq \mathcal{H}^1(K)$ for any admissible competitor $K\subset\R^2$, with equality if and only if $K$ in contained in $\overline\Omega_0$.  Obviously $\boldsymbol{\pi}_0(K)$ is still an admissible competitor, and we infer that any solution of the Steiner problem \eqref{classStein} is contained~$\overline\Omega_0$. Hence, 
\begin{multline}\label{idSteinpb}
\mathscr{S}(\{a_0\}\cup{\rm spt}\,\mu)=\min\Big\{\mathcal{H}^1(K): K\subset\overline\Omega_0 \text{ compact connected,}\\
 K\supset\{a_0\}\cup{\rm spt}\,\mu \Big\} <\infty\,,
\end{multline}
and existence easily follows from Blaschke and Golab theorems (see e.g. \cite{AT}). 
\vskip3pt

The proof of Theorem \ref{thmmain2} departs from the results in \cite{BLS}. The first ingredient is the following lower estimate taken from \cite[Lemma 3.1]{BLS}.

 \begin{lemma}[\cite{BLS}]\label{lemliminf}
 Let $\{v_k\}_{k\in\mathbb{N}}\subset 1+H^1_0(\Omega)\cap C^0(\overline\Omega)$ satisfying $0\leq v_k\leq 1$, and 
 \begin{equation}\label{aprestliminf}
 \sup_{k\in\mathbb{N}}\left(\eps_k\int_\Omega|\nabla v_k|^2\,dx+\frac{1}{4\eps_k}\int_\Omega(1-v_k)^2\,dx+\frac{1}{\alpha_k}\int_{\overline\Omega_0}{\bf D}(v_k;a_0,x)\,d\mu\right)<\infty\,,
 \end{equation}
 for some sequence  $\alpha_k\to0$  of positive numbers. Assume that the sequence $x\mapsto {\bf D}(v_k;a_0,x)$ converges uniformly on $\overline\Omega_0$ to some function ${\bf d}_*:\overline\Omega_0\to [0,\infty)$. Then, $K_*:=\{{\bf d_*}=0\}$ is a compact connected subset of $\overline\Omega_0$ containing $\{a_0\}\cup{\rm spt}\,\mu$, and 
\begin{equation}\label{liminfrecseq}
\mathcal{H}^1(K_*)\leq \liminf_{k\to\infty}\left(\eps_k\int_\Omega|\nabla v_k|^2\,dx+\frac{1}{4\eps_k}\int_\Omega(1-v_k)^2\,dx\right)\,.
\end{equation}
\end{lemma}

The second ingredient is an explicit construction of a ``recovery sequence'' showing the sharpness of the previous lemma. The construction  is provided by \cite[Lemma 2.8]{BLS} (see also \cite{ATo1}) that we (slightly) reformulate as 
 
 \begin{lemma}[\cite{BLS}]\label{lemlimsup}
 Let $K\subset \overline\Omega_0$ be a compact connected set containing $\{a_0\}\cup{\rm spt}\,\mu$ and such that $\mathcal{H}^1(K)<\infty$. 
 There exists a sequence $\{\varphi_k\}_{k\in\mathbb{N}}\subset H^1(\Omega)\cap C^{0}_c(\Omega)$ satisfying $\varphi_k=1$ on~$K$, and 
 \begin{equation}\label{limsuprecseq}
 \limsup_{k\to\infty} \left(\eps_k\int_\Omega|\nabla\varphi_k|^2\,dx +\frac{1}{4\eps_k}\int_\Omega|\varphi_k|^2\,dx\right)\leq \mathcal{H}^1(K)\,.
 \end{equation}
 \end{lemma}
 
\begin{remark}\label{H1finitenergbd} 
 As we shall see below, Lemmas \ref{lemliminf} \& \ref{lemlimsup} imply that assumption $\mathcal{H}^1({\rm spt}\,\mu)<\infty$ is necessary and sufficient to ensure that the minimum value of $ F^\mu_\eps$ over $1+H^1_0(\Omega)$ remains bounded as $\eps\downarrow 0$.  
 \end{remark}

\begin{proof}[Proof of Theorem \ref{thmmain2}]
{\it Step 1.} As discussed above, our assumption $\mathcal{H}^1({\rm spt}\,\mu)<\infty$ implies  $\mathscr{S}(\{a_0\}\cup{\rm spt}\,\mu)<\infty$. Now, given an arbitrary compact connected $K\subset\overline\Omega_0$ containing $\{a_0\}\cup{\rm spt}\,\mu$ and such that $\mathcal{H}^1(K)<\infty$, we consider the sequence $\{\varphi_k\}_{k\in\mathbb{N}}$ provided by Lemma \ref{lemlimsup}, and we set $v_k:=1-\varphi_k\in 1+H^1_0(\Omega)\cap C^0(\overline\Omega)$. We claim that 
\begin{equation}\label{contrrecseq}
\int_{\overline\Omega_0}{\bf D}(\delta_{\eps_k}+v_k^2;a_0,x)\,d\mu \leq \delta_{\eps_k}\mathcal{H}^1(K)\|\mu\|\,.
\end{equation}
Indeed, since $K$ is connected and $\mathcal{H}^1(K)<\infty$,  \cite[Theorem~4.4.7]{AT} yields the existence for every $x\in{\rm spt}\,\mu$ of a curve $\gamma_x\in\mathscr{P}(a_0,x)$ such that $\Gamma(\gamma_x)\subset K$. Since $v_k=0$ on $K$, we deduce that 
$${\bf D}(\delta_{\eps_k}+v_k^2;a_0,x)\leq \int_{\Gamma(\gamma_x)} (\delta_{\eps_k}+v_k^2)\,d\mathcal{H}^1= \delta_{\eps_k}\mathcal{H}^1(\Gamma(\gamma_x))\leq  \delta_{\eps_k}\mathcal{H}^1(K)\quad\forall x\in{\rm spt}\,\mu\,.$$
Integrating this inequality with respect to $\mu$ leads to \eqref{contrrecseq}. Since $\delta_{\eps_k}/\lambda_{\eps_k}\to 0$, we infer from \eqref{limsuprecseq} and \eqref{contrrecseq} that $\limsup_{k}F^\mu_{\eps_k}(v_k)\leq \mathcal{H}^1(K)$. 
On the other hand, $F^\mu_{\eps_k}(u_k)\leq F^\mu_{\eps_k}(v_k)$ by minimality of $u_k$, and we deduce that $\limsup_{k}F^\mu_{\eps_k}(u_k)\leq \mathcal{H}^1(K)$. From the arbitrariness of $K$ and \eqref{idSteinpb}, we conclude that 
\begin{equation}\label{bornsupuk}
\limsup_{k\to\infty}F^\mu_{\eps_k}(u_k)\leq \mathscr{S}(\{a_0\}\cup{\rm spt}\,\mu)<\infty\,.
\end{equation}
\vskip3pt

\noindent{\it Step 2.} Since $0\leq u_k\leq 1$, the sequence  $x\mapsto {\bf D}(\delta_{\eps_k}+u^2_k;a_0,x)$ is a sequence of $(1+\delta_{\eps_k})$-Lipschitz functions on $\overline\Omega_0$, all vanishing at the point $a_0$. By the Arzel\`a-Ascoli Theorem, we can find a (not relabeled) subsequence such that $x\mapsto {\bf D}(\delta_\eps+u^2_k;a_0,x)$ converges uniformly on $\overline\Omega_0$ to some function ${\bf d}_*:\overline\Omega_0\to [0,\infty)$.

Let us now set $\alpha_k:=\lambda_{\eps_k}/(2\sqrt{\delta_{\eps_k}})$. Since $\delta_{\eps_k}=\lambda^\beta_{\eps_k}$ with $\beta\in(1,2)$, we have $\alpha_k\to 0$. Noticing that 
$2\sqrt{\delta_{\eps_k}}u_k\leq \delta_{\eps_k}+u^2_k$,  we have $2\sqrt{\delta_{\eps_k}}{\bf D}(u_k;a_0,x)\leq {\bf D}(\delta_{\eps_k}+u^2_k;a_0,x)$ for every $x\in\overline\Omega_0$. 
In view of \eqref{bornsupuk}, we conclude that 
\begin{equation}\label{intermestasy}
\eps_k\int_\Omega|\nabla u_k|^2\,dx+\frac{1}{4\eps_k}\int_\Omega(1-u_k)^2\,dx+\frac{1}{\alpha_k}\int_{\overline\Omega_0}{\bf D}(u_k;a_0,x)\,d\mu\leq F^\mu_{\eps_k}(u_k)\leq C\,, 
\end{equation}
 for some constant $C$ independent of $k$. By Lemma \ref{lemliminf}, the compact set $K_*:=\{{\bf d}_*=0\}$ is connected and  contains $\{a_0\}\cup{\rm spt}\,\mu$. Gathering  \eqref{liminfrecseq}, \eqref{bornsupuk}, and \eqref{intermestasy} yields
 $$\mathcal{H}^1(K_*)\leq \liminf_{k\to\infty} F^\mu_{\eps_k}(u_k)\leq \limsup_{k\to\infty}  F^\mu_{\eps_k}(u_k)\leq  \mathscr{S}(\{a_0\}\cup{\rm spt}\,\mu)\,.$$
 Therefore, $\mathcal{H}^1(K_*)= \mathscr{S}(\{a_0\}\cup{\rm spt}\,\mu)$ (i.e., $K_*$ solves the Steiner problem relative to $\{a_0\}\cup{\rm spt}\,\mu$), 
 and $F^\mu_{\eps_k}(u_k)\to \mathcal{H}^1(K_*)$. 
 \vskip3pt
 
 \noindent{\it Step 3.} For a radius $r\in (0,\boldsymbol{\eta}_0/2)$ (where $\boldsymbol{\eta}_0$ is given in Lemma \ref{estigradhessloin}), we denote by $V_r$ the open tubular neighborhood of $K_*$ of radius $r$.  Since $K_*\subset\overline\Omega_0$, we have $\overline V_{r/2}\subset\overline V_r\subset \Omega$. We claim that for every $r\in(0,\boldsymbol{\eta}_0/2)$ there exists $k_0(r)\in\mathbb{N}$ such that for every $k\geq k_0(r)$,
 \begin{equation}\label{eqdsVr}
 -\eps_k^2\Delta u_k=\frac{1}{4}(1-u_k) \quad \text{in $\mathscr{D}^\prime(\Omega\setminus \overline V_{r/2})$}\,.
 \end{equation}
To establish \eqref{eqdsVr}, we first invoke the continuity of ${\bf d}_*$ to find $\tau_r>0$ such that $\{{\bf d}_*< 3\tau_r\}\subset V_{r/2}$. Since  $x\mapsto {\bf D}(\delta_{\eps_k}+u^2_k;a_0,x)$ converges uniformly to ${\bf d}_*$, we can find $k_1(r)\in\mathbb{N}$ such that 
 \begin{equation}\label{tardtard2212bis}
 \Big\{x\in\overline\Omega_0: {\bf D}(\delta_{\eps_k}+u^2_k;a_0,x)\leq2\tau_r\Big\}\subset  \{{\bf d}_*< 3\tau_r\} \subset V_{r/2}\quad\forall k\geq k_1(r)\,.
 \end{equation}
On the other hand, since $x\mapsto {\bf D}(\delta_{\eps_k}+u^2_k;a_0,x)$ converges uniformly to $0$ on $K_*\supset {\rm spt}\,\mu$, we can find $k_2(r)\in\mathbb{N}$ such that 
 \begin{equation}\label{tardtard2212}
 {\rm spt}\,\mu\subset  \Big\{x\in\overline\Omega_0: {\bf D}(\delta_{\eps_k}+u^2_k;a_0,x)\leq\tau_r\Big\}\quad\forall k\geq k_2(r)\,.
 \end{equation}
Set $k_0(r):=\max(k_1(r),k_2(r))$, and let us prove that for $k\geq k_0(r)$, 
\begin{multline}\label{longsent}
\text{\sl for all }x\in {\rm spt}\,\mu\;\text{\sl and all }\kappa\in(0,\tau_r)\,,\; 
\text{\sl there exists }\gamma^\kappa_{x}\in\mathscr{P}(a_0,x)\;\text{\sl satisfying}\\
\Gamma(\gamma_x^\kappa)\subset V_{r/2}\;\text{\sl and }\int_{\Gamma(\gamma_x^\kappa)}(\delta_{\eps_k}+u_k^2)\,d\mathcal{H}^1\leq {\bf D}(\delta_{\eps_k}+u^2_k;a_0,x)+\kappa\,. 
\end{multline}
Obviously, for $x\in {\rm spt}\,\mu$ and $\kappa\in(0,\tau_r)$ given, we can find $\gamma^\kappa_{x}\in\mathscr{P}(a_0,x)$ satisfying the second condition, and it suffices to check that 
$\Gamma(\gamma_x^\kappa)\subset V_{r/2}$. Fix $y\in \Gamma(\gamma_x^\kappa)$, and consider $\theta_y\in[0,1]$ such that $\gamma_x^\kappa(\theta_y)=y$. Setting 
$\widetilde \gamma_y(t):=\gamma_x^\kappa(t\theta_y)$, we have $\widetilde\gamma_y\in\mathscr{P}(a_0,y)$ and  $\Gamma(\widetilde\gamma_y)\subset \Gamma(\gamma_x^\kappa)$. Consequently, 
\begin{multline*}
{\bf D}(\delta_{\eps_k}+u^2_k;a_0,y)\leq \int_{\Gamma(\widetilde\gamma_y)}(\delta_{\eps_k}+u_k^2)\,d\mathcal{H}^1\\
\leq \int_{\Gamma(\gamma_x^\kappa)}(\delta_{\eps_k}+u_k^2)\,d\mathcal{H}^1\leq {\bf D}(\delta_{\eps_k}+u^2_k;a_0,x)+\tau_r\leq 2\tau_r\,,
\end{multline*}
by \eqref{tardtard2212}. In view of \eqref{tardtard2212bis}, we have $y\in V_{r/2}$. Hence $\Gamma(\gamma_x^\kappa)\subset V_{r/2}$, and \eqref{longsent} is proved. 

From now on, we assume that $k\geq k_0(r)$. Fix an arbitrary $\varphi\in\mathscr{D}(\Omega\setminus \overline V_{r/2})$, $t\in\mathbb{R}\setminus\{0\}$, and set $w_k:=u_k+t\varphi$. 
Since $w_k=u_k$ in $V_{r/2}$, we infer from \eqref{longsent} that for every $x\in{\rm spt}\,\mu$, 
\begin{multline*}
{\bf D}(\delta_{\eps_k}+w_k^2;a_0,x)\leq \int_{\Gamma(\gamma_x^\kappa)}(\delta_{\eps_k}+w_k^2)\,d\mathcal{H}^1\\
=\int_{\Gamma(\gamma_x^\kappa)}(\delta_{\eps_k}+u_k^2)\,d\mathcal{H}^1 \leq  {\bf D}(\delta_{\eps_k}+u^2_k;a_0,x)+\kappa\quad  \forall\kappa\in(0,\tau_r)\,.
\end{multline*}
Letting $\kappa\downarrow 0$ leads to ${\bf D}(\delta_{\eps_k}+w_k^2;a_0,x)\leq   {\bf D}(\delta_{\eps_k}+u^2_k;a_0,x)$ for every $x\in{\rm spt}\,\mu$. Therefore, 
\begin{equation}\label{matin}
\int_{\overline\Omega_0} {\bf D}(\delta_{\eps_k}+w_k^2;a_0,x)\,d\mu\leq  \int_{\overline\Omega_0} {\bf D}(\delta_{\eps_k}+u_k^2;a_0,x)\,d\mu\,.
\end{equation}
By minimality of $u_k$ we have $F^\mu_{\eps_k}(w_k)-F^\mu_{\eps_k}(u_k)\geq 0$, and  inserting \eqref{matin} in this inequality leads to 
$$2t\eps_k\int_\Omega\nabla u_k\nabla\varphi\,dx+\frac{t}{2\eps_k}\int_\Omega(1-u_k)\varphi\,dx+t^2\eps_k\int_{\Omega} |\nabla\varphi|^2\,dx+\frac{t^2}{\eps^2}\int_{\Omega} |\varphi|^2\,dx \geq 0\,.$$
Dividing this inequality by $t$, and letting $t\downarrow 0$ and $t\uparrow 0$ yields
$$2\eps_k\int_\Omega \nabla u_k\nabla\varphi\,dx+\frac{1}{2\eps_k}\int_\Omega(1-u_k)\varphi\,dx=0\,,$$
and \eqref{eqdsVr} is proved. 
\vskip3pt

\noindent{\it Step 4.} Let us fix $r\in (0,\boldsymbol{\eta}_0/2)$. From \eqref{eqdsVr} and standard elliptic regularity, we infer that $u_k\in C^\infty(\overline\Omega\setminus\overline V_{r/2})$ whenever $k\geq k_0(r)$. Then, arguing as in Lemma \ref{comp}, we derive from  \eqref{eqdsVr} that for $k\geq k_0(r)$,
\begin{equation}\label{estiexpgen}
0\leq 1-u_k(x)\leq \exp\Big(-C_r/\eps_k\Big)\qquad\forall x\in\Omega\setminus V_{3r/4} \,,
\end{equation}
for some constant $C_r>0$ independent of $\eps_k$. Inserting estimate \eqref{estiexpgen} in  \eqref{eqdsVr}, we deduce as in Lemma \ref{estigradhessloin} that for  $k\geq k_0(r)$,
$$\eps_k|\nabla u_k|+\eps_k^2|\nabla^2u_k |\leq C_{r,\boldsymbol{\eta_0}}\exp \Big(-C^\prime_r/\eps_k\Big)\quad\text{in $\Omega\setminus V_{r}$} \,,$$
for some constants $C_{r,\boldsymbol{\eta_0}}$ and $C^\prime_r>0$ independent of $\eps_k$. Hence $u_k\to 1$ in $C^2(\overline\Omega\setminus V_r)$. 
\vskip3pt

\noindent{\it Step 5.} Let us fix $t\in(0,1)$, and show that $\{u_k\leq t\}\to K_*$ in the Hausdorff  sense.  To this purpose, we fix a radius $r>0$. From Step 4 above, we first deduce that $\{u_k\leq t\}\subset V_r$ whenever $k$ is large enough. Before going further, notice that $\{u_k\leq t\}\not=\emptyset$ for $k$ large. Indeed, if $\{u_k\leq t\}=\emptyset$ for infinitely many $k$'s, then 
$$\int_{\overline\Omega_0}{\bf D}(\delta_{\eps}+u_k^2;a_0,x)\,d\mu\geq t^2\int_{\overline\Omega_0}|x-a_0|\,d\mu \quad\text{for infinitely many $k$'s}\,.$$
Since ${\rm spt}\,\mu$ is not reduced to $\{a_0\}$, the right hand side does not vanish, while the left goes to $0$ as $k\to\infty$ by \eqref{bornsupuk}, a contradiction. 

We now denote by $W^k_r$ the open tubular neighborhood of $\{u_k\leq t\}$ of radius $r$. We aim to show that $K_*\subset W^k_r$ for $k$ sufficiently large. Assume by contradiction that for some subsequence $\{k_j\}$, we have $K_*\not\subset W^{k_j}_r$. Then we can find a sequence $\{x_j\}\subset K_*$ such that $x_j\not\in W^{k_j}_r$ for every $j\in\mathbb{N}$. Extracting a subsequence if necessary, we can assume that $x_j\to x_*$ for some point $x_*\in K_*$. Since $\{u_{k_j}\leq t\}\subset \Omega$, by Blaschke's theorem we can also assume that $\{u_{k_j}\leq t\}\to S_t$ in the Hausdorff  sense for some compact set $S_t$. Then ${\rm dist}(x_*, S_t)\geq r$, and we cand find  $j_0(r)\in\mathbb{N}$ such that $\overline B(x_*,r/2)\cap \{u_{k_j}\leq t\}=\emptyset$ for $j\geq j_0(r)$. 
We now distinguish two cases. 

\noindent{\it Case 1.} If $x_*\not=a_0$, set $\tau:=1/2\min(r,|x_*-a_0|)$. Then for every $\gamma\in\mathscr{P}(a_0,x_*)$ we can find $t_\gamma\in(0,1)$ such that $\gamma(t_\gamma)\in\partial B(x_*,\tau)$ and $\gamma([t_\gamma,1])\subset \overline B(x_*,\tau)$. Consequently, for $j\geq j_0(r)$ we have 
$$\int_{\Gamma(\gamma)}(\delta_{\eps_{k_j}}+u^2_{k_j})\,d\mathcal{H}^1 \geq t^2\mathcal{H}^1\big(\gamma([t_\gamma,1])\big)\geq t^2\tau\quad\forall \gamma\in\mathscr{P}(a_0,x_*)\,. $$
In particular ${\bf D}(\delta_{\eps_{k_j}}+u^2_{k_j};a_0,x_*)\geq t^2\tau$ for $j\geq j_0(r)$. Letting $j\to\infty$ yields ${\bf d}_*(x_*)\geq t^2\tau$ which contradicts the fact $x_*\in K_*:=\{{\bf d}_*=0\}$. 

\noindent{\it Case 2.} Assume that $x_*=a_0$. Then the same argument as in Case 1 (applied to $x\in{\rm spt}\,\mu$ instead of $x_*$) 
shows that if $j\geq j_0(r)$, then 
$${\bf D}(\delta_{\eps_{k_j}}+u^2_{k_j};a_0,x)\geq \frac{t^2}{2}\min(r,|x-a_0|)\quad\forall x\in{\rm spt}\,\mu\,. $$
Since ${\rm spt}\,\mu$ is not reduced to $\{a_0\}$ by assumption, we have for $j\geq j_0(r)$, 
$$\int_{\overline\Omega_0} {\bf D}(\delta_{\eps_{k_j}}+u^2_{k_j};a_0,x)\,d\mu\geq \frac{t^2}{2}\int_{\overline\Omega_0} \min(r,|x-a_0|)\,d\mu>0\,. $$
Once again, the left hand side of this inequality goes to $0$ as $j\to\infty$ by \eqref{bornsupuk}, which provides the desired contradiction. 
\vskip3pt

\noindent{\it Step 6.} To complete the proof of Theorem \ref{thmmain2}, it only remains to show that ${\bf d}_*(x)={\rm dist}(x,K_*)$. Since $K_*:=\{{\bf d}_*=0\}$, we only have to show this identity for $x\not\in K_*$. First, since ${\bf d}_*$ is a $1$-Lipschitz function (as pointwise limite of $(1+\delta_{\eps_k})$-Lipschitz functions), we obviously have ${\bf d}_*(x)\leq {\rm dist}(x,K_*)$. Now fix a point $x\in\overline\Omega_0\setminus K_*$, an arbitrary $\tau \in\big(0,{\rm dist}(x,K_*)\big)$, and an arbitrary $t\in(0,1)$. We infer from Step 5 that 
$u_k^2\geq t^2$ in $\overline B(x,\tau)$ for $k$ large enough. Then, arguing as in Step 5, Case 1, we obtain ${\bf D}(\delta_{\eps_{k}}+u^2_{k};a_0,x)\geq t^2\tau$ for $k$ large enough. 
Letting $k\to\infty$ yields ${\bf d}_*(x)\geq t^2\tau$. From the arbitrariness of $\tau$ and $t$, we conclude that ${\bf d}_*(x)\geq {\rm dist}(x,K_*)$. 
\end{proof} 
 
 \begin{remark}\label{rembad}
 In the spirit of Proposition \ref{limitweak}, one can study the asymptotic behavior of minimizers of $F_\eps^{\mu_\eps}$ over $1+H^1_0(\Omega)$, for some sequence of measures  $\mu_\eps\mathop{\rightharpoonup}\limits^*\mu$ as $\eps\to0$, and eventually varying base points $a_0^\eps\to a_0$. In this general setting, it is necessary  to assume that 
$\sup_{\eps\in(0,1)} F_\eps^{\mu_\eps}(u_\eps)<\infty$, where $u_\eps$ denotes a minimizer of $F_\eps^{\mu_\eps}$ over $1+H^1_0(\Omega)$. Since \cite[Lemma 3.1]{BLS} actually allows for such $\eps$-dependence in the a priori estimate \eqref{aprestliminf}, Steps 1 \& 2 in the previous proof carry over. Hence, up to a subsequence, $x\mapsto {\bf D}(\delta_\eps+u^2_\eps;a_0^\eps,x)$ converges uniformly on $\overline\Omega_0$  as $\eps\to 0$ to some $1$-Lipschitz function ${\bf d}_*$, the compact set $K_*:=\{{\bf d}_*=0\}$  is connected and $\{a_0\}\cup{\rm spt}\,\mu\subset K_*$. Then, $K_*$ solves the Steiner problem relative to $\{a_0\}\cup{\rm spt}\,\mu$, and $F_\eps^{\mu_\eps}(u_\eps)\to\mathcal{H}^1(K_*)$.  
   
 If we assume that 
\begin{equation}\label{condcvspt}
\text{${\rm spt}\,\mu_\eps\to {\rm spt}\,\mu$ in the Hausdorff  sense}\,,
\end{equation} 
 then (all) the other conclusions of Theorem \ref{thmmain2} remain. The argument follows essentially the  same lines as above. 
Note that \eqref{condcvspt} includes the case where $\mu_\varepsilon$ is a discrete approximation of $\mu$ as in Lemma \ref{approxmeas}. 
 
On the other hand, if one drops condition \eqref{condcvspt}, then Hausdorff convergence of sublevel sets of minimizers can fail (their Hausdorff limit can be different from any Steiner set relative to $\{a_0\}\cup{\rm spt}\,\mu$). To illustrate this fact, let us consider the following example. Let $a_0,a_1,a_2\in \overline \Omega_0$ be three distinct points such that $a_1\in(a_0,a_2)$, and set $\mu_\kappa:=\delta_{a_0}+\delta_{a_1}+\kappa\delta_{a_2}$ with $\kappa\in[0,1]$. For each $\kappa>0$, the segment $[a_0,a_2]$ is the unique solution of the Steiner problem \eqref{classStein} relative to $\mu_{\kappa}$, while $[a_0,a_1]$ is the unique solution relative to $\mu_0$. Obviously, $\mu_{\kappa}\mathop{\rightharpoonup}\limits^*\mu_0$ as $\kappa\downarrow0$, but ${\rm spt}\,\mu_\kappa=\{a_0,a_1,a_2\}\not\to {\rm spt}\,\mu_0=\{a_0,a_1\}$. Now, consider two sequences $\kappa_j\downarrow0$ and $\eps_n\downarrow 0$, and for each $(j,n)\in\mathbb{N}^2$, a minimizer $u_{j,n}\in 1+H^1_0(\Omega)$ of $F_{\eps_n}^{\mu_{\kappa_j}}$ (with base point $a_0$). By Theorem \ref{thmmain2}, $\{u_{n,j}\leq 1/2\}\to [a_0,a_2]$ in the Hausdorff  sense as $n\to\infty$ for every $j\in\mathbb{N}$. Consequently, we can find a subsequence $\{n_j\}$ such that $\{u_{n_j,j}\leq 1/2\}\to [a_0,a_2]$ in the Hausdorff  sense as $j\to\infty$. 
\end{remark}

 \subsection{Towards the average distance and optimal compliance problems}\label{asymptavdopc}
 
 In this last subsection, we discuss the asymptotic behavior as $\eps\to 0$ of the functionals $\mathscr{F}^\eps_{\rm avd}$ and $\mathscr{F}^\eps_{\rm opc}$ defined in  
\eqref{defpfavd} and \eqref{defpfopc}, and of their minimizers. For this purpose, it is more convenient to consider the reduced functionals $\widetilde{\mathscr{F}}^\eps_{\rm avd}:\overline\Omega_0\times\mathscr{M}(\overline\Omega_0;\R^2)\to [0,\infty]$ and $\widetilde{\mathscr{F}}^\eps_{\rm opc}:\overline\Omega_0\times L^2(\Omega_0;\R^2)\to [0,\infty]$ given by 
$$\widetilde{\mathscr{F}}^\eps_{\rm avd}(a_0,v):=\min_{u\in 1+H^1_0(\Omega)\cap L^\infty(\Omega)} \mathscr{F}^\eps_{\rm avd}(a_0,v,u)\,,$$
and 
$$\widetilde{\mathscr{F}}^\eps_{\rm opc}(a_0,v):=\min_{u\in 1+H^1_0(\Omega)\cap L^\infty(\Omega)} \mathscr{F}^\eps_{\rm opc}(a_0,v,u)\,.$$
By Theorem \ref{mainmain}, for every $(a_0,v)\in \overline\Omega_0\times\mathscr{M}(\overline\Omega_0;\R^2)$, respectively every $(a_0,v)\in \overline\Omega_0\times L^2(\Omega_0;\R^2)$, there exists $u_\eps=u_\eps(a_0,v)\in 1+H^1_0(\Omega)\cap L^\infty(\Omega)$ such that 
$$\widetilde{\mathscr{F}}^\eps_{\rm avd}(a_0,v)= \mathscr{F}^\eps_{\rm avd}(a_0,v,u_\eps)\text{ or } \widetilde{\mathscr{F}}^\eps_{\rm avd}(a_0,v)=\mathscr{F}^\eps_{\rm opc}(a_0,v,u_\eps)\,.$$
Assuming that \eqref{deppara} holds, Theorem \ref{thmmain2} and Remark \ref{H1finitenergbd} then imply that $\widetilde{\mathscr{F}}^\eps_{\rm avd}$ and  $\widetilde{\mathscr{F}}^\eps_{\rm opc}$ converge pointwise as $\eps\to0$ to $\mathscr{F}_{\rm avd}$ and  $\mathscr{F}_{\rm opc}$, respectively. 

Beyond this pointwise convergence, one can reproduce the proof of \cite[Theorem 5.7]{BLS} (using assumption \eqref{deppara} as in Step 2 of the proof of Theorem \ref{thmmain2}) to show that $\widetilde{\mathscr{F}}^\eps_{\rm avd}$ actually $\Gamma$-converges to $\mathscr{F}_{\rm avd}$ (for the ($\overline\Omega_0\times$weak*)-topology), and $\widetilde{\mathscr{F}}^\eps_{\rm opc}$ $\Gamma$-converges to $\mathscr{F}_{\rm opc}$ (for the ($\overline\Omega_0\times$weak)-topology). In addition, if $\{(a_0^\eps,v_\eps)\}_{\eps>0}$ is a recovery sequence of a configuration $(a_0,v)$ of finite energy, and $\widetilde{\mathscr{F}}^\eps_{\rm avd}(a^\eps_0,v_\eps)= \mathscr{F}^\eps_{\rm avd}(a^\eps_0,v_\eps,u_\eps)$  or    $\widetilde{\mathscr{F}}^\eps_{\rm opc}(a^\eps_0,v_\eps)= \mathscr{F}^\eps_{\rm opc}(a^\eps_0,v_\eps,u_\eps)$, then $F_\eps^{\mu(v_\eps)}(u_\eps)\to \mathscr{S}(\{a_0\}\cup{\rm spt}\,\mu(v))$ as $\eps\to 0$, and the sequence $x\mapsto{\bf D}(\delta_\eps+u_\eps^2;a_0^\eps,x)$ converges uniformly on $\overline\Omega_0$ to some function ${\bf d_*}$. The set $K_*:=\{{\bf d_*}=0\}$ is connected, $\{a_0\}\cup{\rm spt}\,\mu(v)\subset K_*$, and $\mathcal{H}^1(K_*)= \mathscr{S}(\{a_0\}\cup{\rm spt}\,\mu(v))$, see Remark \ref{rembad}. 

The same consideration applies  in case $(a_0^\eps,v_\eps,u_\eps)$ is a minimizer of either $\mathscr{F}^\eps_{\rm avd}$ or $\mathscr{F}^\eps_{\rm opc}$. By $\Gamma$-convergence,  $(a_0^\eps,v_\eps)$ (sub)-converges as $\eps\to0$ to a minimizer $(a_0^\sharp,v^\sharp)$ of $\mathscr{F}_{\rm avd}$ or $\mathscr{F}_{\rm opc}$, respectively. 
Consequently, $K_*=K^\sharp_{\rm avd}$ or $K_*=K^\sharp_{\rm opc}$ as in \eqref{avion1}--\eqref{avion2}, i.e., $K_*$ solves the average distance problem or the optimal compliance problem, respectively. To conclude, one may wonder wether or not the sublevel sets $\{u_\eps\leq t\}$ Hausdorff converge to $K_*$, as in Theorem~\ref{thmmain2}. In view of Remark \ref{rembad}, this question remains quite unclear, and it certainly requires a specific analysis taking full advantage of the minimality of the pair $(v_\eps,u_\eps)$.

\section*{Acknowledgments}
The authors have been supported by the {\sl Agence Nationale de la Recherche} through the grants  ANR-12-BS01-0014-01 (Geometrya) and ANR-14-CE25-0009-01 (MAToS), and by  the PGMO research projects MACRO and COCA.


\end{document}